\documentclass[secthm,seceqn,amsthm,ussrhead,reqno,12pt]{amsart}
\usepackage[utf8]{inputenc}
\usepackage[english]{babel}
\usepackage[symbol]{footmisc}
\usepackage{amssymb,amsmath,amsthm,amsfonts,xcolor,enumerate,hyperref, comment,longtable, cleveref}
\usepackage{verbatim}
\usepackage{times}
\usepackage{cite}
\usepackage{pdflscape}
\usepackage{ulem}
\usepackage[mathcal]{euscript}
\usepackage{tikz}
\usepackage{hyperref}
\usepackage{cancel}
\usepackage{stmaryrd}
 
\usepackage{amsfonts}
\usepackage{amssymb}
\usepackage{times}
\usepackage{xcolor}
\usepackage{tkz-graph}
\usepackage{url}
\usepackage{float}
\usepackage{tasks}
\usepackage{cite}
\usepackage{hyperref}

\usepackage{amsfonts}
\usepackage{amsmath}
\usepackage{eurosym}
\usepackage{geometry}

\usepackage{caption,booktabs}

\theoremstyle{plain}
\newtheorem{theorem}{Theorem}
\newtheorem{lemma}[theorem]{Lemma}
\newtheorem{proposition}[theorem]{Proposition}
\newtheorem{remark}[theorem]{Remark}

\newtheorem{definition}[theorem]{Definition}
\newtheorem{corollary}[theorem]{Corollary}

\usetikzlibrary{arrows}

\newtheorem*{theoremA1}{Theorem A1}
\newtheorem*{theoremA2}{Theorem A2}
\newtheorem*{theoremA3}{Theorem A3}

\newtheorem*{theoremG1}{Theorem G1}
\newtheorem*{theoremG2}{Theorem G2}
\newtheorem*{theoremG3}{Theorem G3}

\usepackage{fouriernc}
\begin{document}
 
 \bigskip

\noindent{\Large
The algebraic and geometric classification of\\  noncommutative Jordan algebras}\footnote{The authors thank Amir Fernández Ouaridi for sharing some useful software programs.}\footnote{
The  first part of the work is supported by 
FCT   UIDB/00212/2020 and UIDP/00212/2020;
grant FZ-202009269, Ministry of Higher Education, Science and Innovations of the Republic of Uzbekistan.
The second part of this work is supported by the Russian Science Foundation under grant 22-11-00081.
}

 \bigskip

\begin{center}

 {\bf
Hani Abdelwahab\footnote{Department of Mathematics, 
Faculty of Science, Mansoura University,  Mansoura, Egypt; \ haniamar1985@gmail.com},
Kobiljon Abdurasulov\footnote{CMA-UBI, University of  Beira Interior, Covilh\~{a}, Portugal;  \ 
Saint Petersburg  University, Russia; \ Institute of Mathematics Academy of
Sciences of Uzbekistan, Tashkent, Uzbekistan; \ abdurasulov0505@mail.ru} 
\&   
Ivan Kaygorodov
\footnote{CMA-UBI, University of  Beira Interior, Covilh\~{a}, Portugal; \    kaygorodov.ivan@gmail.com}

}

\end{center}

\ 

\noindent {\bf Abstract:}
{\it  
In this paper, we develop a method to obtain the algebraic classification of
noncommutative Jordan algebras from the classification of Jordan algebras of the same dimension.
We use this method to obtain the algebraic classification of complex $3$-dimensional noncommutative Jordan algebras. 
As a byproduct, we obtain the classification
of complex $3$-dimensional Kokoris, standard,
generic Poisson, and generic Poisson--Jordan algebras;
and also complex $4$-dimensional nilpotent Kokoris and standard algebras. 
In addition, we consider the geometric
classification of   varieties of cited algebras,
that is the description of its irreducible components. 
 }

 \bigskip 


\bigskip 

 \
 

	 \bigskip

%%%%%%%%%%%%%%%%%%%%%%%%%%%%%%%%%%%%%%%%%%%%%%%%%%%%%%%%%%%%%%%%%%%%%
  
%\newpage]]

%\tableofcontents 

\section*{ Introduction}
  
The  algebraic classification (up to isomorphism) of algebras of small dimensions from a certain variety defined by a family of polynomial identities is a classic problem in the theory of non-associative algebras. 
Another interesting approach to studying algebras of a fixed dimension is to study them from a geometric point of view (that is, to study the degenerations and deformations of these algebras). The results in which the complete information about degenerations of a certain variety is obtained are generally referred to as the geometric classification of the algebras of these varieties. There are many results related to the algebraic and geometric classification of 
Jordan, Lie, Leibniz, Zinbiel, and other algebras 
(see,  \cite{     BC99, kppv,               EM22, BT22,    gkk,  GRH, ha16,           ikv18,       ikv20,              fkkv} and references in \cite{k23,l24,MS}). 
 The geometric classification of algebras from a certain variety is based on the notion of degeneration, that is a ''dual'' notion to deformations \cite{b,c}. 

Noncommutative Jordan algebras were introduced by Albert in \cite{Alb}. He noted that the structure theories of alternative and Jordan algebras share so many nice properties that it is natural to conjecture that these algebras are members of a more general class with a similar theory. So he introduced the variety of noncommutative Jordan algebras defined by the Jordan identity and the flexibility identity. 
Namely, the variety of noncommutative Jordan algebras is defined by the following identities:
\[
\begin{array}{rcl}
x(yx) &=&  (xy)x,\\
x^2(yx) &=&  (x^2y)x\end{array} \]
The class of noncommutative Jordan algebras turned out to be vast: for example, apart from alternative and Jordan algebras, it contains quasi-associative and quasi-alternative algebras, 
quadratic flexible algebras and all anticommutative algebras. However, the structure theory of this class is far from being nice.
Nevertheless, certain progress was made in the study of the structure theory of noncommutative Jordan algebras and superalgebras 
 (see, for example \cite{ccr22,Sch,McC,McC2,KLP17,popov,spa,k58,k60,O64,p2v,ps19, sh71,pcm}). 
So, 
Schafer gave the first structure theory for noncommutative Jordan algebras of characteristic zero in \cite{Sch}. Kleinfeld and  Kokoris proved that 
a simple, flexible, power-associative algebra of finite dimension over a field of characteristic zero is a noncommutative Jordan algebra  \cite{Kleinfeld}.
A coordinatization theorem for noncommutative Jordan algebras was obtained in a paper by McCrimmon \cite{McC}.
Later, 
McCrimmon shows that the property of being a noncommutative Jordan algebra is preserved under homotopy \cite{McC2}.
Noncommutative Jordan algebras with additional identity 
$([x,y],y,y)=0$ were studied in papers by Shestakov and Schafer \cite{sh71,S94}.
Strongly prime noncommutative Jordan algebras were studied by 
Skosyrskiĭ \cite{sk89,sk91}.
A connection between  noncommutative Jordan algebras and  $(-1,-1)$-balanced Freudenthal Kantor triple systems was established by 
Elduque,   Kamiya, and Okubo \cite{EKO}. Cabrera Serrano,  Cabrera García, and   Rodríguez Palacios studied 
the algebra of multiplications of prime and semiprime noncommutative Jordan algebras \cite{spa, ccr22}. Jumaniyozov, Kaygorodov, and  Khudoyberdiyev classified complex $4$-dimensional nilpotent noncommutative Jordan algebras \cite{ikv18}.
Connections between noncommutative Jordan algebras and conservative algebras are discussed in \cite{pcm}.
There are also some results in noncommutative Jordan superalgebras. Namely, Pozhidaev and Shestakov classified all simple finite-dimensional 
noncommutative Jordan superalgebras in \cite{ps19};
Kaygorodov, Popov, and Lopatin studied the structure of simple noncommutative Jordan superalgebras \cite{KLP17};
Popov described representations of simple noncommutative Jordan superalgebras in \cite{popov, p2v}.

The main goal of the present paper is to obtain the algebraic and geometric description of the variety of complex $3$-dimensional noncommutative Jordan algebras. To do so, we first determine all such $3$-dimensional algebra structures, up to isomorphism (what we call the algebraic classification), and then proceed to determine the geometric properties of the corresponding variety, namely its dimension and description of the irreducible components (the geometric classification). 
As some corollaries, we have the algebraic and geometric classification of complex $3$-dimensional Kokoris and standard algebras.

\medskip 

Our main results are summarized below.

\begin{theoremA1}
There are infinitely many isomorphism classes of   complex 
$3$-dimensional  noncommutative Jordan algebras, described explicitly in 
    Theorems \ref{asscom},\ref{3-dim Jordan},\ref{3ant} and \ref{3-dim noncommutative Jordan}
    in terms of 
$9$ one-parameter families, 
$1$ two-parameter family 
and 
$23$ additional isomorphism classes.
There are only $18$ non-isomorphic complex $4$-dimensional nilpotent (non-$2$-step nilpotent)
noncommutative Jordan algebras listed in Theorem \ref{nilp4nj}.

\end{theoremA1}

\begin{theoremA2}
There are infinitely many isomorphism classes of   complex 
$3$-dimensional Kokoris  algebras, 
described explicitly in Theorem \ref{3-dim Kokoris algebra}  in terms of 
$3$ one-parameter families
and 
$19$ additional isomorphism classes.
There are only $10$ non-isomorphic complex $4$-dimensional nilpotent (non-$2$-step nilpotent)
Kokoris algebras listed in Theorem \ref{koknilp}.
\end{theoremA2}

\begin{theoremA3}
There are infinitely many isomorphism classes of   complex 
$3$-dimensional standard  algebras, 
described explicitly in Theorem \ref{teostan}  in terms of 
$1$ one-parameter family
and 
$32$ additional isomorphism classes.
There are only $18$ non-isomorphic complex $4$-dimensional nilpotent (non-$2$-step nilpotent)
standard   algebras listed in Theorem \ref{nilp4nj}.
\end{theoremA3}

From the geometric point of view, in many cases, the irreducible components of the variety are determined by the rigid algebras, i.e., algebras whose orbit closure is an irreducible component. It is worth mentioning that this is not always the case and already in \cite{f68}, Flanigan had shown that the variety of complex $3$-dimensional nilpotent associative algebras has an irreducible component which does not contain any rigid algebra  --- it is instead defined by the closure of a union of a one-parameter family of algebras. Here, we encounter a different situation. Informally, although Theorems~G1, ~G2 and~G3 show that there are   {\it generic} algebras and 
  {\it generic} parametric families in the variety of $3$-dimensional  
  noncommutative Jordan, Kokoris, and standard algebras. 

\begin{theoremG1}
The variety of complex $3$-dimensional   noncommutative Jordan algebras has dimension $9$. 
It is defined by  $4$ rigid algebras and three one-parametric families of algebras and can be described as the closure of the union of $\mathrm{GL}_3(\mathbb{C})$-orbits of the algebras given in Theorem \ref{geo3}.
The variety of complex $4$-dimensional nilpotent noncommutative Jordan algebras has dimension $14$. 
It is defined by  $3$ rigid algebras and two one-parametric families of algebras and can be described as the closure of the union of $\mathrm{GL}_4(\mathbb{C})$-orbits of the algebras given in \cite[Theorem B]{ikv18}.
\end{theoremG1}

\begin{theoremG2}
The variety of complex $3$-dimensional   Kokoris algebras has dimension $9$. 
It is defined by  $3$ rigid algebras and two one-parametric families of algebras and can be described as the closure of the union of $\mathrm{GL}_3(\mathbb{C})$-orbits of the algebras given in Theorem \ref{geo1}.
The variety of complex $4$-dimensional nilpotent Kokoris algebras has dimension $13$. 
It is defined by  $3$ rigid algebras and two one-parametric families of algebras and can be described as the closure of the union of $\mathrm{GL}_4(\mathbb{C})$-orbits of the algebras given in Theorem \ref{geo1nilp}.
\end{theoremG2}

\begin{theoremG3}
The variety of complex $3$-dimensional   standard   algebras has dimension $9$. 
It is defined by  $13$ rigid algebras and one one-parametric family of algebras and can be described as the closure of the union of $\mathrm{GL}_3(\mathbb{C})$-orbits of the algebras given in Theorem \ref{geo2}.
The variety of complex $4$-dimensional nilpotent standard    algebras has dimension $14$. 
It is defined by  $3$ rigid algebras and two one-parametric families of algebras and can be described as the closure of the union of $\mathrm{GL}_4(\mathbb{C})$-orbits of the algebras given in \cite[Theorem B]{ikv18}.
\end{theoremG3}
% \newpage

 \section{The algebraic classification of non-associative algebras}

In our paper, we are working with finite-dimensional vector spaces over the complex field.

\subsection{The  classification method of noncommutative Jordan algebras}

Let $({\bf A}, \cdot)$ be an algebra. We consider the following two new products
on the underlying vector space ${\bf A}$ defined by 
\begin{center}
$x\circ y :=\frac{1}{2}(x\cdot y+y\cdot x), \hspace{2cm} [x,y] :=\frac{1}{2}%
(x\cdot y-y\cdot x).$
\end{center}
Let us denote $\textbf{A}^+:= (\textbf{A}, \circ),$ \ 
$\textbf{A}^-:=(\textbf{A},
[\cdot,\cdot])$ 
and the associator as $(x,y,z):=(x\cdot y)\cdot z-x\cdot (y\cdot z).$

Recall that an algebra $({\bf A}, \cdot)$ is called a  {noncommutative Jordan algebra}
if it satisfies the identities: 
\begin{equation}
(x , y, x)  = 0,  \label{flexible}
\end{equation}
\begin{equation}
(x\cdot x,   y,   x) = 0.  \label{Jordan}
\end{equation}

Identity (\ref{flexible}) is called the flexibility identity. A flexible algebra 
${\bf A}$ satisfies the Jordan identity (\ref{Jordan})  if and only if
the corresponding   algebra $\textbf{A}^{+}$ is a Jordan algebra
(i.e.,  a commutative algebra with identity (\ref{Jordan})). The class of noncommutative Jordan algebras is extremely large. It contains
all Jordan and alternative algebras, as well as all anticommutative
algebras.

\begin{definition}
An algebra $({\bf P}, \circ, [\cdot,\cdot])$ is called  a  generic Poisson--Jordan  algebra (resp., generic Poisson  algebra)\footnote{To the best of our knowledge, such “associative-commutative”  algebras were introduced independently by Cannas and Weinstein  under the name “almost Poisson algebras” and by Shestakov  under the name “general Poisson algebras” (later changed into “generic Poisson algebras” in 
 a paper by  Kolesnikov,   Makar-Limanov, and Shestakov). We will use the last terminology.} if 
  $({\bf P}, \circ)$ is a Jordan (resp., associative commutative) algebra, 
   $({\bf P},  [\cdot,\cdot])$ is an anticommutative algebra and 
these two operations are required to satisfy the following   identity: 
\begin{equation*}
[ x\circ y,z] =[ x,z] \circ y+x\circ [ y,z].
\end{equation*}%

\end{definition}

The class of generic Poisson--Jordan algebras is extremely extensive. It contains all Poisson algebras, Malcev--Poisson algebras, and Malcev--Poisson--Jordan algebras, as well as all generic Poisson algebras.

\begin{proposition}
$({\bf A},\cdot) $ is a noncommutative Jordan algebra if and only
if $({\bf A},\circ ,[\cdot,\cdot])$ is a generic Poisson--Jordan  algebra.
\end{proposition}

\begin{proof}
Since the flexible law is equivalent to $[x\circ y,z] = x\circ [y,z] +
y\circ [x,z]$, the proof is finished.
\end{proof}

\begin{definition}
\label{def_homomorphism}\textrm{Let $({\bf P}_1,\circ_1,[\cdot,\cdot]_1)$
and $({\bf P}_2,\circ_2,[\cdot,\cdot]_2)$ be two
algebras. A   linear map $\phi :%
{\bf P}_1 \to {\bf P}_2$ is a {homomorphism} if it is preserving the products, that is, 
\begin{equation*}
\phi (x\circ_1 y) =\phi(x) \circ_2 \phi(y), \hspace{2cm} \phi([x,y]_1) =
[\phi(x),\phi(y)]_2.
\end{equation*}
 }
\end{definition}

Let $({\bf A}_1,\cdot_1)$ and $({\bf A}_2,\cdot_2)$ be two
noncommutative Jordan algebras and let $({\bf A}_1,\circ_1,[\cdot,%
\cdot]_1)$ and $({\bf A}_2,\circ_2,[\cdot,\cdot]_2)$ be its associated
generic Poisson--Jordan algebras. If $({\bf A}_1,\cdot_1)$ and $({\bf A}_2,\cdot_2)$
are isomorphic, then  it is easy to see that 
the generic Poisson--Jordan algebras $({\bf A}_1,\circ_1,[\cdot,\cdot]_1)$ and $({\bf A}_2,\circ_2,[\cdot, \cdot]_2)$ are isomorphic. Conversely, we can show that if the generic
Poisson--Jordan algebras $({\bf A}_1,\circ_1,[\cdot,\cdot]_1)$ and $(
{\bf A}_2,\circ_2,[\cdot,\cdot]_2)$ are isomorphic, then the
noncommutative Jordan algebras $({\bf A}_1,\cdot_1)$ and $({\bf A}%
_2,\cdot_2)$ are isomorphic. So we have the following result:

\begin{proposition}
Every generic Poisson--Jordan  algebra 
(resp., generic Poisson algebra) $({\bf A},\circ, [\cdot,\cdot])$ is
associated with precisely one noncommutative Jordan (resp., Kokoris\footnote{About Kokoris algebras, see subsection \ref{Kokoris}}) algebra $({\bf A}%
,\cdot)$. That is, we
have a bijective correspondence between generic Poisson--Jordan (resp., generic Poisson) algebras
  and
noncommutative Jordan (resp., Kokoris) algebras.
\end{proposition}

\begin{definition}
\textrm{Let $({\bf A},\circ)$ be a Jordan algebra. Let ${\bf Z}^2({\bf A},%
{\bf A})$ be the set of all skew symmetric bilinear maps $\theta :%
{\bf A}\times {\bf A} \to {\bf A}$ such that 
\begin{equation*}
\theta(x\circ y,z) =\theta(x,z) \circ y +x\circ \theta(y,z).
\end{equation*} }
\end{definition}

\noindent  For $\theta \in {\bf Z}^2(%
{\bf A},{\bf A})$ we define on ${\bf A}$ a product $%
[\cdot,\cdot]_{\theta} :{\bf A}\times {\bf A}\to {\bf A}$ by 
\begin{equation}
[x,y]_{\theta} := \theta(x,y).  \label{producttheta}
\end{equation}

\begin{lemma}
Let $({\bf A},\circ)$ be a Jordan algebra and $\theta \in {\bf Z}^2({\bf A}%
,{\bf A})$. Then $( {\bf A},\circ,[\cdot,\cdot]_{\theta})$ is a
generic Poisson--Jordan  algebra endowed with the product defined in \eqref{producttheta}
and $( {\bf A},\cdot_{\theta})$ is a
noncommutative Jordan algebra, 
where $x\cdot_{\theta} y := x \circ y + [x,y]_{\theta}.$ 

\end{lemma}
In the reverse direction, if $({\bf A},\circ ,[\cdot,\cdot])$ is a
generic Poisson--Jordan  algebra, then there exists $\theta \in {\bf Z}^2({\bf A}%
, {\bf A})$ such that $({\bf A},\circ , [\cdot,\cdot]_{\theta}) $
and $({\bf A},\circ ,[\cdot,\cdot])$ are isomorphic. Indeed, consider
the skew symmetric bilinear map $\theta :{\bf A} \times {\bf A} \to 
{\bf A}$ defined by $\theta(x,y) := [x,y].$ 
 Then $\theta \in {\bf Z}^2({\bf A},{\bf A})$ and $({\bf A}, \circ ,
[\cdot,\cdot]_{\theta}) = ({\bf A},\circ , [\cdot,\cdot])$.

\medskip

Now, let $({\bf A},\circ)$ be a Jordan algebra and $\mathrm{{Aut}(%
{\bf A})}$ be the automorphism group of ${\bf A}$ with respect to
product $\circ$. Then $\mathrm{{Aut}({\bf A})}$ acts on ${\bf Z}^2({\bf A},%
{\bf A})$ by 
\begin{equation*}
(\theta \ast \phi)(x,y) := \phi^{-1}\bigl(\theta \bigl(\phi(x),\phi(y) \bigr)%
\bigr),
\end{equation*}
where $\phi \in \mathrm{{Aut}({\bf A})}$ and $\theta \in {\bf Z}^2({\bf A}%
, {\bf A})$.

\begin{lemma}
Let $({\bf A},\circ)$ be a Jordan algebra and $\theta, \vartheta \in {\bf Z}^2(%
{\bf A},{\bf A})$. Then the generic Poisson--Jordan algebras $(%
{\bf A},\circ ,[\cdot,\cdot]_{\theta})$ and $({\bf A},\circ
,[\cdot,\cdot]_{\vartheta})$ are isomorphic if and only if there exists $%
\phi \in \mathrm{{Aut}({\bf A})}$ satisfying $\theta \ast \phi
=\vartheta $.
\end{lemma}

\begin{proof}
Suppose $\phi : ({\bf A},\circ ,[\cdot,\cdot]_{\vartheta}) \to ({\bf A},\circ , [\cdot,\cdot]_{\theta})$ is an isomorphism of generic 
Poisson--Jordan algebras. Then $\phi \in \text{Aut}({\bf A})$, and by Definition 
\ref{def_homomorphism}, we have $\phi([x,y]_{\vartheta}) =
[\phi(x),\phi(y)]_{\theta}$. That is, $%
\phi(\vartheta(x,y)) = \theta(\phi(x),\phi(y))$.
Hence $\vartheta(x,y) = \phi^{-1}\bigl(\theta\bigl(\phi(x), \phi(y)\bigr)%
\bigr) = (\theta \ast \phi)(x,y)$.

To prove the converse, if $\theta \ast \phi =\vartheta$ then $\phi : (%
{\bf A},\circ ,[\cdot,\cdot]_{\vartheta}) \to ({\bf A},\circ ,
[\cdot,\cdot]_{\theta})$ is an isomorphism since $\phi(\vartheta(x,y)) =
\theta (\phi(x),\phi(y))$.
\end{proof}

Hence, we have a procedure to classify the generic Poisson--Jordan algebras
(and therefore   noncommutative Jordan algebras) associated with a given
Jordan algebra $({\bf A},\circ)$. It consists of three steps:

\begin{enumerate}
\item[{\bf Step} $1.$] Compute ${\bf Z}^2({\bf A},{\bf A})$.

\item[{\bf Step} $2.$] Find the orbits of $\mathrm{{Aut}({\bf A})}$ on ${\bf Z}^2(%
{\bf A},{\bf A})$.

\item[{\bf Step} $3.$] Choose a representative $\theta$ from each orbit and then
construct the generic Poisson--Jordan  algebra $({\bf A},\circ ,
[\cdot,\cdot]_{\theta})$ (the noncommutative Jordan algebra $({\bf A}%
,\cdot_{\theta})$).
\end{enumerate}

Let us introduce the following notations. Let $\{e_1,\dots,e_n\}$ be a fixed basis of a Jordan algebra $({\bf A},\circ)$. Define 
$\mathrm{\Lambda }^2({\bf A},\mathbb{C})$ to be the space of all skew
symmetric bilinear forms on ${\bf A}$, that is, \begin{center}
    $\mathrm{\Lambda}^2(%
{\bf A},\mathbb{C}) := 
\langle \Delta_{i,j} | 1\leq i<j \leq n\rangle$,
\end{center} where $\Delta_{i,j}$ is the skew-symmetric bilinear form $%
\Delta_{i,j}:{\bf A}\times {\bf A} \to \mathbb{C}$ defined by 
\begin{equation*}
\Delta_{i,j}( e_l,e_m) :=\left\{ 
\begin{tabular}{rl}
$1,$ & if $(i,j) = (l,m),$ \\ 
$-1,$ & if $(i,j) = (m,l),$ \\ 
$0,$ & otherwise.%
\end{tabular}
\right.
\end{equation*}
Now, if $\theta \in {\bf Z}^2({\bf A},{\bf A})$ then $\theta$ can be
uniquely written as $\theta (x,y) = \sum_{i=1}^n B_i(x,y)e_i$, where $%
B_1,\dots , B_n$ are skew symmetric bilinear forms on ${\bf A}$. Also,
we may write $\theta = (B_{1},\dots ,B_n)$. Let $\phi^{-1}\in \mathrm{Aut}(%
{\bf A})$ be given by the matrix $( b_{ij})$. If $(\theta \ast
\phi)(x,y) = \sum_{i=1}^n B_i^{\prime }(x,y)e_i$, then $B_i^{\prime }=
\sum_{j=1}^n b_{ij}\phi^t B_j\phi$, whenever $i \in \{1,\dots,n\}$.

\begin{remark}
\label{[1 0]}Let $X=%
\begin{pmatrix}
\alpha & \beta%
\end{pmatrix}%
\in {\mathbb{M}}_{1\times 2}({\mathbb{C}})$ and $X\neq 0$. Then there
exists an invertible matrix $A\in {\mathbb{M}}_{2\times 2}({\mathbb{C}})$
such that $XA=%
\begin{pmatrix}
1 & 0%
\end{pmatrix}%
$. To see this, suppose first that $\alpha \neq 0$. Then $%
\begin{pmatrix}
\alpha & \beta%
\end{pmatrix}%
\begin{pmatrix}
\alpha ^{-1} & -\beta \\ 
0 & \alpha%
\end{pmatrix}%
=%
\begin{pmatrix}
1 & 0%
\end{pmatrix}%
$. Assume now that $\alpha =0$. Then $%
\begin{pmatrix}
0 & \beta%
\end{pmatrix}%
\begin{pmatrix}
0 & 1 \\ 
\beta ^{-1} & 0%
\end{pmatrix}%
=\allowbreak 
\begin{pmatrix}
1 & 0%
\end{pmatrix}%
$.
\end{remark}

\subsection{The algebraic classification of 
noncommutative Jordan algebras}

Let us mention that by $\circ$ we will denote a commutative multiplication 
and by $[\cdot,\cdot]$ we denote an anticommutative multiplication.
All products that can be recuperated by commutativity or anticommutativity will be omitted.

\subsubsection{The algebraic classification of 
complex $3$-dimensional noncommutative Jordan algebras}

\begin{theorem}(see, \cite{gkk})
\label{asscom}Let ${\bf A}$ be a complex $3$-dimensional associative commutative algebra. Then ${\bf A}$ is isomorphic to
one  of the following algebras:

\begin{longtable}{lllllllllll}
${\bf A}_{01} $ & $:$ & $ e_{1}\circ e_{1}=e_{2}$ \\
${\bf A}_{02}^{0}$ & $:$ & $e_{1}\circ e_{2}=e_{3}$ \\
${\bf A}_{03}$ & $:$ & $e_{1}\circ e_{1}=e_{2}$ & $e_{1}\circ e_{2}=e_{3}$\\
${\bf A}_{04} $ & $:$ & $e_{1}\circ e_{1}=e_{1}$ & $ e_{2}\circ e_{2}=e_{2}$ & $e_{3}\circ e_{3}=e_{3}$\\
${\bf A}_{05} $ & $:$ & $e_{1}\circ e_{1}=e_{1}$ & $e_{2}\circ e_{2}=e_{2}$ & $e_{2}\circ e_{3}=e_{3}$\\
${\bf A}_{06}$ & $:$ & $e_{1}\circ e_{1}=e_{1} $ & $e_{1}\circ e_{2}=e_{2}$ & $e_{1}\circ e_{3}=e_{3}$\\
${\bf A}_{07}$ & $:$& $e_{1}\circ e_{1}=e_{1}$ & $e_{1}\circ e_{2}=e_{2}$ & $e_{1}\circ e_{3}=e_{3}$ & $e_{2}\circ e_{2}=e_{3}$\\
${\bf A}_{08}$ & $:$ & $e_{1}\circ e_{1}=e_{1}$ & $e_{2}\circ e_{2}=e_{2}$\\

${\bf A}_{09}$ & $:$ & $e_{1}\circ e_{1}=e_{1}$ & $e_{1}\circ e_{2}=e_{2}$\\

${\bf A}_{10}$ & $:$ & $e_{1}\circ e_{1}=e_{1}$\\

${\bf A}_{11}$ & $:$ & $e_{1}\circ e_{1}=e_{1}$ & $e_{2}\circ e_{2}=e_{3}$\\

\end{longtable}

\end{theorem}

\begin{theorem}(see, \cite{gkk})
\label{3-dim Jordan}Let ${\bf A}$ be a complex $3$-dimensional Jordan algebra. Then ${\bf A}$ is an associative commutative algebra listed in Theorem \ref{asscom} or isomorphic to
one  of the following algebras:

\begin{longtable}{lllllllllll}
${\bf A}_{12}$ & $:$ & $e_{1}\circ e_{1}=e_{1}$ & $e_{2}\circ e_{2}=e_{2}$ & $ e_{3}\circ e_{3}=e_{1}+e_{2}$ & $e_{1}\circ e_{3}=
\frac{1}{2}e_{3}$ & $e_{2}\circ e_{3}=\frac{1}{2}e_{3}$\\

${\bf A}_{13}^{0}$ & $:$ & $e_{1}\circ e_{1}=e_{1}$ & $e_{1}\circ e_{2}=\frac{1}{2}
e_{2}$ & $e_{1}\circ e_{3}=e_{3}$\\

${\bf A}_{14}^{0,0}$ & $:$ & $e_{1}\circ e_{1}=e_{1}$ & $e_{1}\circ e_{2}=\frac{1}{2}
e_{2}$ & $e_{1}\circ e_{3}=\frac{1}{2}e_{3}$\\

${\bf A}_{15} $ & $:$ & $e_{1}\circ e_{1}=e_{1}$ & $e_{2}\circ e_{2}=e_{3}$ & $e_{1}\circ e_{2}=
\frac{1}{2}e_{2}$\\

${\bf A}_{16}$ & $:$ & $e_{1}\circ e_{1}=e_{1}$ & $e_{2}\circ e_{2}=e_{3}$ & $e_{1}\circ e_{2}=
\frac{1}{2}e_{2}$ & $e_{1}\circ e_{3}=e_{3}$\\

${\bf A}_{17}^{0}$ & $:$ & $e_{1}\circ e_{1}=e_{1}$ & $e_{2}\circ e_{2}=e_{2}$ & $e_{1}\circ e_{3}=
\frac{1}{2}e_{3}$ & $e_{2}\circ e_{3}=\frac{1}{2}e_{3}$\\

${\bf A}_{18}^{0}$ & $:$ & $e_{1}\circ e_{1}=e_{1}$ & $e_{1}\circ e_{2}=\frac{1}{2}e_{2}$\\

${\bf A}_{19}^{0}$ & $:$ & $e_{1}\circ e_{1}=e_{1}$ & 
$e_{2}\circ e_{2}=e_{2}$ & $e_{1}\circ e_{3}=\frac{1}{2}e_{3}$  \\
\end{longtable}

\end{theorem}

\begin{theorem}(see, \cite{ikv20})\label{3ant}
 Let ${\bf A}$ be a complex $3$-dimensional 
anticommutative algebra. Then ${\bf A}$ is isomorphic to one of the
following algebras:

\begin{longtable}{llllll}
${\bf A}_{20}$ & $:$ & $ [ e_{2},e_{3}] =e_{1}$\\

${\bf A}_{21}$ & $:$ & $[ e_{1},e_{3}] =e_{1}$ & $[
e_{2},e_{3}] =e_{2}$\\

${\bf A}_{22}^{\alpha }$ & $:$ & $[e_{1},e_{3}]=e_{1}+e_{2}$ & $[ e_{2},e_{3}] =\alpha e_{2}$\\

${\bf A}_{23}$ & $:$ & $[ e_{1},e_{2} ]  =e_{3}$ & $[
e_{1},e_{3} ] =-e_{2}$ & $[ e_{2},e_{3} ] =e_{1}$\\

${\bf A}_{24}^{\alpha }$ & $:$ & $ [  e_{1},e_{2} ] =e_{3}$ & $ [
e_{1},e_{3} ] =e_{1}+e_{3}$ & $[ e_{2},e_{3} ] =\alpha e_{2}$\\

${\bf A}_{25}$ & $:$ & $[ e_{1},e_{2} ] =e_{1}$ & $ [
e_{2},e_{3} ] =e_{2}$\\

${\bf A}_{26}$ & $:$ & $[ e_{1},e_{2} ] =e_{3}$ & $[
e_{1},e_{3} ] =e_{1}$ & $[ e_{2},e_{3}] =e_{2}$\\
\end{longtable}

All listed algebras are non-isomorphic except:
 ${\bf A}_{22}^{\alpha }\cong {\bf A}_{22}^{\alpha^{-1} },$  ${\bf A}_{24}^{\alpha }\cong {\bf A}_{24}^{\alpha^{-1} }$.

\end{theorem}

\begin{theorem}
\label{3-dim noncommutative Jordan}Let ${\bf A}$ be a 
complex $3$-dimensional
(noncommutative and  non-anticommutative) noncommutative Jordan algebra. Then $\bf{A}$
is isomorphic to one of the following algebras:
{\small
 \begin{longtable}{llllllllll}
${\bf A}_{02}^{\alpha\neq0}$ & $:$ & $e_1\cdot e_2=\left( 1+\alpha \right) e_3$ & $e_2\cdot e_1=\left(1-\alpha \right) e_3$ \\

${\bf A}_{13}^{\alpha\neq0 }$ & $:$ & $e_{1}\cdot e_{1}=e_{1}$ & $e_{1}\cdot e_{2}=(%
\frac{1}{2}+\alpha )e_{2}$ & $e_{2}\cdot e_{1}=(\frac{1}{2}-\alpha)e_{2}$ \\
&  & $e_{1}\cdot e_{3}=e_{3}$  & $e_{3}\cdot e_{1}=e_{3}$ \\

${\bf A}_{14}^{(\alpha,\beta)\neq(0,0)}$ & $:$ & $e_{1}\cdot e_{1}=e_{1}$ & $e_{1}\cdot e_{2}=(\frac{1}{2}+\alpha )e_{2}$ & $e_{2}\cdot e_{1}=(\frac{1}{2}-\alpha)e_{2}$  \\
&& $e_{1}\cdot e_{3}=(\frac{1}{2}+\beta )e_{3}$ & $e_{3}\cdot e_{1}=(\frac{1}{2}-\beta)e_{3}$   \\

${\bf A}_{17}^{\alpha\neq0 }$ & $:$ & $e_{1}\cdot e_{1}=e_{1}$ & $e_{2}\cdot e_{2}=e_{2}$ & $e_{1}\cdot e_{3}=(\frac{1}{2}+\alpha )e_{3}$  \\
&& $e_{3}\cdot e_{1}=(
\frac{1}{2}-\alpha )e_{3}$ & $e_{2}\cdot e_{3}=(\frac{1}{2}-\alpha
)e_{3}$ &  $e_{3}\cdot e_{2}=(\frac{1}{2}+\alpha )e_{3}$ \\

${\bf A}_{18}^{\alpha \neq0}$ & $:$ & $e_{1}\cdot e_{1}=e_{1}$ & $e_{1}\cdot e_{2}=(\frac{1}{2}+\alpha )e_{2}$ & $e_{2}\cdot e_{1}=(\frac{1}{2}-\alpha )e_{2}$ \\

${\bf A}_{19}^{\alpha \neq0}$ & $:$ & $e_{1}\cdot e_{1}=e_{1}$ & $e_{2}\cdot e_{2}=e_{2}$ & $e_{1}\cdot e_{3}=(\frac{1}{2}+\alpha )e_{3}$ & $e_{3}\cdot e_{1}=(\frac{1}{2}-\alpha )e_{3}$   \\

${\bf A}_{27}$ & $:$ & $e_1\cdot e_1=e_2$ & $e_1\cdot e_3=e_3$ & $e_3\cdot e_1=-e_3$ \\

${\bf A}_{28}$ & $:$ & $e_1\cdot e_1=e_2$ & $ e_1\cdot e_3=e_2$ & $e_3\cdot e_1=-e_2$ \\

${\bf A}_{29}$ & $:$ & $e_{1}\cdot e_{1}=e_{1}$ & $e_{1}\cdot e_{2}=e_{2}$ & $e_{2}\cdot e_{1}=e_{2}$ & $e_{1}\cdot e_{3}=e_{3}$ \\
&& $e_{3}\cdot e_{1}=e_{3}$ & $e_{2}\cdot e_{3}=e_{3}$ & $e_{3}\cdot e_{2}=-e_{3}$ \\

${\bf A}_{30}$ & $:$ & $e_{1}\cdot e_{1}=e_{1}$ & $e_{2}\cdot e_{3}=e_{3}$ & $e_{3}\cdot e_{2}=-e_{3}$ \\

${\bf A}_{31}^{\alpha }$ & $:$ & $e_{1}\cdot e_{1}=e_{1}$ & $e_{1}\cdot e_{2}=(\frac{1}{2}+\alpha )e_{2}$ & $e_{2}\cdot e_{1}=(\frac{1}{2}-\alpha)e_{2}$  \\
& & $e_{1}\cdot e_{3}=e_{2}+(\frac{1}{2}+\alpha )e_{3}$ & $e_{3}\cdot e_{1}=-e_{2}+(\frac{1}{2}-\alpha )e_{3}$ \\

${\bf A}_{32}$ & $:$ & $e_{1}\cdot e_{1}=e_{1}$ & $e_{1}\cdot e_{2}=\frac{1}{2}e_{2}+e_{3}$ & $e_{2}\cdot e_{1}=\frac{1}{2}e_{2}-e_{3}$ & $e_{1}\cdot e_{3}=\frac{1%
}{2}e_{3}$ \\
& & $e_{3}\cdot e_{1}=\frac{1}{2}e_{3}$ & $e_{2}\cdot e_{3}=e_{2}$ & $e_{3}\cdot e_{2}=-e_{2}$\\

${\bf A}_{33}^{\alpha }$ & $:$ & $e_{1}\cdot e_{1}=e_{1}$ & $e_{1}\cdot e_{2}=\frac{1}{2}e_{2}$ & $e_{2}\cdot e_{1}=\frac{1}{2}e_{2}$ \\
&& $e_{1}\cdot e_{3}=(\frac{1}{2}+\alpha)e_{3}$  & $e_{3}\cdot e_{1}=(\frac{1}{2}-\alpha )e_{3}$ \\&& $e_{2}\cdot e_{3}=e_{2}$ & $e_{3}\cdot e_{2}=-e_{2}$\\

\end{longtable} }

All listed algebras are non-isomorphic except:
 ${\bf A}_{02}^{\alpha }\cong {\bf A}_{02}^{-\alpha}$, ${\bf A}_{17}^{\alpha }\cong {\bf A}_{17}^{-\alpha}$ and ${\bf A}_{14}^{\alpha,\beta}\cong {\bf A}_{14}^{\beta,\alpha}$. 
\end{theorem}

\begin{proof}
By Theorem \ref{3-dim Jordan}, we may assume ${\bf A}^+ \in \left\{ 
{\bf A}_{01},{\bf A}_{02}^{0},\ldots ,{\bf A}_{19}^0\right\} $. So, we are interested in studying   $\theta \in {\bf Z}^{2}( {\bf A}^+,{\bf A}^+),$
such that $\theta \neq 0.$
We have the following cases.

\begin{enumerate}
    \item 
 \underline{${\bf A}^+ \in \left\{ {\bf A}_{03},{\bf A}_{04}, {\bf A}_{05}, {\bf A}_{07}, {\bf A}_{08}, {\bf A}_{09}, {\bf A}_{11}, {\bf A}_{12}, {\bf A}_{15}, {\bf A}_{16}\right\} $}. Then ${\bf Z}^{2}\left( {\bf A}^+,{\bf A}%
^+\right) =\{0\}$. 
In this case, there are no new algebras for us.

\medskip 

  \item  \underline{${\bf A}^+ = {\bf A}_{01}$}. 
  Let $\theta  \in {\bf Z}^2( {\bf A}_{01},{\bf A}_{01})$. Then $\theta
=(0,\alpha \Delta_{1,3},\beta \Delta_{1,3})$ for some $\alpha,\beta \in 
\mathbb{C}$. $\mathrm{Aut}(%
{\bf A}_{01})$ consists of the invertible matrices of the following form: 
\begin{equation*}
\phi = 
\begin{pmatrix}
a_{11} & 0 & 0 \\ 
a_{21} & a_{11}^{2} & a_{23} \\ 
a_{31} & 0 & a_{33}%
\end{pmatrix}
.
\end{equation*}

Let $\phi\in \mathrm{Aut}({\bf A}_{01})$. Then $\theta * \phi =
(0,\alpha^{\prime }\Delta _{1,3},\beta^{\prime }\Delta _{1,3})$ where 
\begin{center}
$\alpha^{\prime }  = a_{11}^{-1}(\alpha a_{33} - \beta a_{23}),$ \  
$\beta^{\prime }=\beta a_{11}.$
\end{center}

Let us distinguish two cases:

\begin{itemize}
    \item  If $\beta \neq 0$, we define $\phi $ to
be the following automorphism:%
\begin{equation*}
\phi = 
\begin{pmatrix}
{\beta^{-1} } & 0 & 0 \\ 
0 & {\beta ^{-2}} & {\alpha }{\beta^{-1} } \\ 
0 & 0 & 1%
\end{pmatrix}
.
\end{equation*}%
Then $\theta \ast \phi =(0,0,\Delta _{1,3})$. So we get  the algebra ${\bf A}_{27}$.

\item If $\beta =0$, we define $\phi $ to be the following
automorphism:%
\begin{equation*}
\phi = 
\begin{pmatrix}
\alpha & 0 & 0 \\ 
0 & \alpha ^{2} & 0 \\ 
0 & 0 & 1%
\end{pmatrix}%
.
\end{equation*}%
Then $\theta \ast \phi =(0,\Delta _{1,3},0)$. So,  we get the algebra ${\bf A}_{28}$.
\end{itemize}

  \medskip     \item  \underline{${\bf A}^+ = {\bf A}_{02}^{0}$}. Let $\theta  \in {\bf Z}^2({\bf A}_{02}^{0},{\bf A}_{02}^{0})$. 
  Then $\theta
= (0,0,\alpha \Delta _{1,2})$ for some $\alpha \in \mathbb{C}$.  
$\mathrm{Aut}( {\bf A}_{02}^{0})$ 
consists of the invertible matrices of the following form:

\begin{equation*}
\phi = 
\begin{pmatrix}
a_{11} & a_{12} & 0 \\ 
a_{21} & a_{22} & 0 \\ 
a_{31} & a_{32} & a_{11}a_{22}+a_{12}a_{21}%
\end{pmatrix},  \text{ where 
}a_{12}=a_{21}=0\text{ or }a_{11}=a_{22}=0.
\end{equation*}
\noindent
Then  $\theta * \phi = (0,0,\alpha^{\prime }\Delta_{1,2})$ and $\alpha^{\prime }= \frac{\alpha
a_{11}a_{22}-\alpha a_{12}a_{21}}{a_{11}a_{22}+a_{12}a_{21}}.$
Since $a_{12}=a_{21}=0\text{ or }a_{11}=a_{22}=0$, we have $(\alpha ^{\prime
})^2=\alpha ^{2}$. So we get the representatives $\theta ^{\alpha }=\left(
0,0,\alpha \Delta _{1,2}\right) $. Moreover, we have $\theta ^{\alpha }$ and $%
\theta ^{\alpha ^{\prime }}$ in the same orbit if and only if $(\alpha
^{\prime})^2=\alpha ^{2}$. 
So, we get a family of algebras ${\bf A}_{02}^{\alpha\neq0 }$.

  \medskip     \item  \underline{${\bf A}^+ = {\bf A}_{06}$}. Let $\theta   \in {\bf Z}^{2}\left( {\bf A}_{06}, {\bf A}_{06}\right) $. Then $\theta =\left( 0,\alpha \triangle
_{2,3},\beta \triangle _{2,3}\right) $ for some $\alpha ,\beta \in \mathbb{C}
$.  $\mathrm{Aut}%
\left( {\bf A}_{06} \right) $  consists of the invertible matrices of the
following form:%
\begin{equation*}
\phi = 
\begin{pmatrix}
1 & 0 & 0 \\ 
0 & a_{22} & a_{23} \\ 
0 & a_{32} & a_{33}%
\end{pmatrix}
.
\end{equation*}%
Then $\theta \ast \phi =\left( 0, (\alpha a_{33}-\beta a_{23})\triangle _{2,3},(\beta a_{22}-\alpha a_{32})\triangle _{2,3}\right) $ and 
for suitable $a_{ij},$ we get the representative $\theta =\left( 0,0,\triangle _{2,3} \right) $
and   the algebra ${\bf A}_{29}$.

%\medskip \underline{${\bf A}^+ \in \left\{ {\bf A}_{07},{\bf A}_{08},{\bf A}_{09}\right\} $}. Then ${\bf Z}^{2}\left( {\bf A}^+,\mathcal{J}% ^+\right) =\{0\}$. So we get the algebras ${\bf A}_{07},{\bf A}_{08}, {\bf A}_{09}$.

  \medskip     \item  \underline{${\bf A}^+ = {\bf A}_{10}$}. Let  $\theta  \in {\bf Z}^{2}\left( 
{\bf A}_{10},{\bf A}_{10}\right) $. Then $\theta =\left( 0,\alpha
\triangle _{2,3},\beta \triangle _{2,3}\right) $ for some $\alpha ,\beta \in 
\mathbb{C}$. $\mathrm{Aut}\left( 
{\bf A}_{10}\right) $ consists of the invertible matrices of the following
form:%
\begin{equation*}
\phi = 
\begin{pmatrix}
1 & 0 & 0 \\ 
0 & a_{22} & a_{23} \\ 
0 & a_{32} & a_{33}%
\end{pmatrix}
.
\end{equation*}%
Then  $\theta \ast \phi =\left( 0, (\alpha a_{33}-\beta a_{23})\triangle _{2,3}, (\beta a_{22}-\alpha a_{32})\triangle _{2,3}\right) $ and 
for suitable $a_{ij},$ we get the representative $\theta =\left( 0,0,\triangle _{2,3}\right) $
and  the algebra ${\bf A}_{30}$.

%\medskip \underline{${\bf A}^+ \in \left\{ {\bf A}_{11},{\bf A}_{12}\right\} $}. Then ${\bf Z}^{2}\left( {\bf A}^+,{\bf A}^+\right)=\{0\} $. Hence we get the algebras ${\bf A}_{11},{\bf A}_{12}$.

  \medskip     \item  \underline{${\bf A}^+ = {\bf A}_{13}^{0}$}. Let $\theta  \in {\bf Z}^{2}\left( {\bf A}_{13}^{0},{\bf A}_{13}^{0}\right) $. Then $\theta =\left( 0,\alpha \Delta
_{1,2},0\right) $ for some $\alpha \in\mathbb{C}$.  $\mathrm{Aut}\left( {\bf A}_{13}^{0}\right) $  consists of the
invertible matrices of the following form:%
\begin{equation*}
\phi = 
\begin{pmatrix}
1 & 0 & 0 \\ 
a_{21} & a_{22} & 0 \\ 
0 & 0 & a_{33}%
\end{pmatrix}%
.
\end{equation*}%
Then $\theta \ast \phi =\left( 0,\alpha ^{\prime }\Delta _{1,2},0\right) $ 
 and $\alpha ^{\prime }=\alpha $. Thus we have the the family of representatives $\theta
^{\alpha }=\left( 0,\alpha \Delta _{1,2},0\right) $ and the family of  algebras ${\bf A}_{13}^{\alpha \neq0}$.

  \medskip     \item  \underline{${\bf A}^+ = {\bf A}_{14}^{0,0}$}. Let $\theta  \in {\bf Z}^{2}\left( {\bf A}_{14}^{0,0},{\bf A}_{14}^{0,0}\right) $. Then 
\begin{equation*}
\theta =\left( 0,\alpha_{1} \Delta _{1,2}+\alpha_{2} \Delta
_{1,3}+\alpha_{3} \Delta _{2,3},\alpha_{4} \Delta _{1,2}+\alpha_{5} \Delta
_{1,3}+\alpha_{6} \Delta _{2,3}\right)
\end{equation*}
for some $\alpha_{1},\ldots,\alpha_{6} \in\mathbb{C}$.  $\mathrm{Aut}\left( {\bf A}_{14}^{0,0}\right) $
consists of the invertible matrices of the following form:%
\begin{equation*}
\phi = 
\begin{pmatrix}
1 & 0 & 0 \\ 
a_{21} & a_{22} & a_{23} \\ 
a_{31} & a_{32} & a_{33}%
\end{pmatrix}
.
\end{equation*}%
Then 
$\theta \ast \phi=\left( 0,\beta_{1} \Delta _{1,2}+\beta_{2} \Delta
_{1,3}+\beta_{3} \Delta _{2,3},\beta_{4} \Delta _{1,2}+\beta_{5} \Delta
_{1,3}+\beta_{6} \Delta _{2,3}\right)$ and
\begin{longtable}{lcl}
$\beta _{1}$ &$=$&$\frac{ (\alpha_{1}a_{22}+\alpha_{2}a_{32}+\alpha_{3}(a_{21}a_{32}
-a_{22}a_{31}))a_{33}
-(\alpha_{4}a_{22}+\alpha_{5}a_{32}+\alpha_{6}(a_{21}a_{32}-a_{22}a_{31}))a_{23}}{\det \phi },$ \\
$\beta_{2}$ &$=$&$\frac{(\alpha_{1}a_{23}+\alpha_{2}a_{33}+
\alpha_{3}(a_{21}a_{33}-a_{31}a_{23}))a_{33}
-(\alpha_{4}a_{23}+\alpha_{5}a_{33} +\alpha_{6}(a_{21}a_{33}-a_{31}a_{23}))a_{23}}{\det \phi },$\\
$\beta_{3}$ &$=$&$\alpha _{3}a_{33}-\alpha _{6}a_{23},$ \\
$\beta_{4}$ &$=$&$\frac{ -(\alpha_{1}a_{22}+\alpha_{2}a_{32}+\alpha_{3}(a_{21}a_{32}
-a_{22}a_{31}))a_{32}
+(\alpha_{4}a_{22}+\alpha_{5}a_{32}+\alpha_{6}(a_{21}a_{32}-a_{22}a_{31}))a_{22}}{\det \phi },$ \\
$\beta_{5}$ &$=$&$\frac{- (\alpha_{1}a_{23}+\alpha_{2}a_{33}+
\alpha_{3}(a_{21}a_{33}-a_{31}a_{23}))a_{32}
+(\alpha_{4}a_{23}+\alpha_{5}a_{33} +\alpha_{6}(a_{21}a_{33}-a_{31}a_{23}))a_{22}}{\det \phi },$\\
$\beta _{6}$ &$=$&$\alpha _{6}a_{22}-\alpha _{3}a_{32}.$
\end{longtable}
\noindent
By Remark \ref{[1 0]}, we may assume $\left( \alpha _{3},\alpha _{6}\right)
\in \left\{ \left( 0,0\right) ,\left( 1,0\right) \right\} $. 
Let us consider these cases separately.
\begin{itemize}
    \item Assume first
that $(\alpha _{3},\alpha _{6})=(0,0)$. Then%
\begin{equation*}
\begin{pmatrix}
\beta _{1} & \beta _{2} \\ 
\beta _{4} & \beta _{5}%
\end{pmatrix}%
=\left( 
\begin{array}{cc}
a_{22} & a_{23} \\ 
a_{32} & a_{33}%
\end{array}%
\right) ^{-1}%
\begin{pmatrix}
\alpha _{1} & \alpha _{2} \\ 
\alpha _{4} & \alpha _{5}%
\end{pmatrix}%
\left( 
\begin{array}{cc}
a_{22} & a_{23} \\ 
a_{32} & a_{33}%
\end{array}%
\right) .
\end{equation*}%
From here, we may assume $%
\begin{pmatrix}
\alpha _{1} & \alpha _{2} \\ 
\alpha _{4} & \alpha _{5}%
\end{pmatrix}%
\in \left\{ 
\begin{pmatrix}
\alpha & 0 \\ 
0 & \beta%
\end{pmatrix}%
,%
\begin{pmatrix}
\alpha & 1 \\ 
0 & \alpha%
\end{pmatrix}%
\right\} $. So we get   representatives
\begin{center}
    $\theta _{1}^{\alpha ,\beta
}=\left( 0,\alpha \Delta _{1,2},\beta \Delta _{1,3}\right) $ and $\theta
_{2}^{\alpha }=\left( 0,\alpha \Delta _{1,2}+\Delta _{1,3},\alpha \Delta
_{1,3}\right) $.
\end{center} Clearly, $\theta _{1}^{\alpha ,\beta }$ and $\theta
_{2}^{\alpha }$ are not in the same orbit. Furthermore, we have $\theta
_{1}^{\alpha ,\beta }$ and $\theta _{1}^{\alpha ^{\prime },\beta ^{\prime }}$
in the same orbit if and only if $\left\{ \alpha ,\beta \right\} =\left\{
\alpha ^{\prime },\beta ^{\prime }\right\} $. Also, $\theta _{2}^{\alpha }$
and $\theta _{2}^{\alpha ^{\prime }}$ are in the same orbit if and only if $%
\alpha =\alpha ^{\prime }$. Thus we get families of algebras ${\bf A}_{14}^{(\alpha,\beta)\neq (0,0) }$ and ${\bf A}_{31}^{\alpha }$. 

\item Assume now that $(\alpha _{3},\alpha _{6})=(1,0)$. Then we have the following cases:

\begin{itemize}
\item $\alpha _{4}\neq 0$. We define $\phi $ to be the following
automorphism:%
\begin{equation*}
\phi =%
\begin{pmatrix}
1 & 0 & 0 \\ 
-\frac{\alpha _{5}^{2}+\alpha _{2}\alpha _{4}}{\alpha _{4}} & \frac{1}{%
\alpha _{4}} & -\frac{\alpha _{5}}{\alpha _{4}} \\ 
\alpha _{1}+\alpha _{5} & 0 & 1%
\end{pmatrix}%
.
\end{equation*}%
Then $\theta \ast \phi =\left( 0,\Delta _{23},\Delta _{12}\right) $. So we
get   the algebra ${\bf A}_{32}$.

\item $\alpha _{4}=0$. We choose $\phi $ to be the following automorphism:%
\begin{equation*}
\phi =%
\begin{pmatrix}
1 & 0 & 0 \\ 
-\alpha _{2} & 1 & 0 \\ 
\alpha _{1} & 0 & 1%
\end{pmatrix}%
.
\end{equation*}%
Then $\theta \ast \phi =\left( 0,\Delta _{23},\alpha _{5}\Delta _{13}\right) 
$. Hence we get the representatives $\theta ^{\alpha }=\left( 0,\Delta
_{23},\alpha\Delta _{13}\right) $. The representatives $\theta ^{\alpha }$
and $\theta ^{\beta }$ are in the same orbit if and only if $\alpha=\beta$.
Thus, we obtain the algebras ${\bf A}_{33}^{\alpha }$.
\end{itemize}

\end{itemize}
%\medskip \underline{${\bf A}^+ \in \left\{ {\bf A}_{15},{\bf A}_{16}\right\} $}. Then ${\bf Z}^{2}\left( {\bf A}^+,{\bf A}^+\right)=\{0\} $. Hence we get the algebras ${\bf A}_{15}, {\bf A}_{16}$.

  \medskip     \item  \underline{${\bf A}^+ = {\bf A}_{17}^{0}$}. Let $\theta  \in  {\bf Z}^{2}\left( {\bf A}_{17}^{0},{\bf A}_{17}^{0} \right) $. Then $\theta =\left( 0,0,\alpha \left( \Delta
_{1,3}-\Delta _{2,3}\right) \right) $ for some $\alpha \in \mathbb{C}$. $\mathrm{Aut}\left( {\bf A}_{17}^{0}\right) $  consists of the invertible matrices of the following form: 
\begin{equation*}
\left( 
\begin{array}{ccc}
1 & 0 & 0 \\ 
0 & 1 & 0 \\ 
a_{31} & -a_{31} & a_{33}%
\end{array}%
\right) \mbox{ and } \left( 
\begin{array}{ccc}
0 & 1 & 0 \\ 
1 & 0 & 0 \\ 
a_{31} & -a_{31} & a_{33}%
\end{array}%
\right) .
\end{equation*}%
Then $\theta \ast \phi =\left( 0,0,\alpha ^{\prime }\left( \Delta _{1,3}-\Delta
_{2,3}\right) \right) $ and     $(\alpha^\prime) ^2=\alpha^{2}$.  
So we get the family of algebras ${\bf A}_{17}^{\alpha\neq0}$.

  \medskip     \item  \underline{${\bf A}^+ = {\bf A}_{18}^{0}$}. Let   $\theta  \in  {\bf Z}^{2}\left( {\bf A}_{18}^{0},{\bf A}_{18}^{0} \right) $. Then $\theta =\left( 0,\alpha \Delta
_{1,2},0\right) $ for some $\alpha \in \mathbb{C}$.   $\mathrm{Aut}\left( {\bf A}_{18}^{0}\right) $  consists of
the invertible matrices of the following form
\begin{equation*}
\phi =\left( 
\begin{array}{ccc}
1 & 0 & 0 \\ 
a_{21} & a_{22} & 0 \\ 
0 & 0 & a_{33}%
\end{array}%
\right).
\end{equation*}%
Then $\theta \ast \phi =\left( 0,\alpha ^{\prime }\Delta _{1,2},0\right) $ and $
\alpha ^{\prime }=\alpha $.   So we get the family of algebras ${\bf A}_{18}^{\alpha\neq0 }$.

  \medskip     \item  \underline{${\bf A}^+ = {\bf A}_{19}^{0}$}. Let  $\theta \in  {\bf Z}^{2}\left( {\bf A}_{19}^{0},{\bf A}_{19}^{0} \right) $. Then $\theta =\left( 0,0,\alpha \Delta_{1,3}\right) $ for some $\alpha \in \mathbb{C}$.  $\mathrm{Aut}\left( {\bf A}_{19}^{0}\right) $  consists of
the invertible matrices of the following form:%
\begin{equation*}
\phi =\left( 
\begin{array}{ccc}
1 & 0 & 0 \\ 
0 & 1 & 0 \\ 
a_{31} & 0 & a_{33}
\end{array}%
\right) .
\end{equation*}%
Then $\theta \ast \phi =\left( 0, 0, \alpha ^{\prime }\Delta _{1,3}\right) $
and  $\alpha ^{\prime }=\alpha $.  So we get the family of algebras ${\bf A}_{19}^{\alpha \neq0}$.

\end{enumerate}
\end{proof}

\subsubsection{The algebraic classification of 
  complex   $4$-dimensional nilpotent noncommutative Jordan  algebras} 
Let us remember the classification of  $4$-dimensional nilpotent noncommutative Jordan algebras
obtained in \cite{ikv18}. This classification will be useful for results given below, in subsections \ref{Kokorisnilp} and \ref{standnilp}.

\begin{theorem}\label{nilp4nj}
Let $\mathcal J$ be a complex $4$-dimensional nilpotent noncommutative Jordan 
algebra. Then $\mathcal J$ is a $2$-step nilpotent  algebra listed in \cite{kppv} or  isomorphic to one
of the following algebras:
\end{theorem}

\begin{longtable}{llllllll}

        ${\mathcal J}_{01}$&$:$& $e_1 e_1 = e_2$ & $e_1 e_2=e_3$  & $e_2 e_1=e_3$ \\
        ${\mathcal J}_{02}$&$:$& $e_1 e_1 = e_2$ & $e_2 e_3=e_4$  & $e_3 e_2=e_4$ & $e_3 e_1=e_4$ \\
        ${\mathcal J}_{03}$&$:$& $e_1 e_1 = e_2$ & $e_1 e_2=e_4$  & $e_2 e_1=e_4$ & $e_3 e_1=e_4$ & $e_3 e_3=e_4$\\
        ${\mathcal J}_{04}$&$:$& $e_1 e_1 = e_2$ & $e_1 e_2=e_4$  & $e_2 e_1=e_4$ & $e_3 e_1=e_4$\\
        ${\mathcal J}_{05}$&$:$& $e_1 e_1 = e_2$ & $e_2 e_3=e_4$  & $e_3 e_2=e_4$\\
        ${\mathcal J}_{06}$&$:$& $e_1 e_1 = e_2$ & $e_1 e_2=e_4$  & $e_2 e_1=e_4$ & $e_3e_3 =e_4$\\
        ${\mathcal J}_{07}$&$:$& $e_1 e_2 = e_3$ & $e_1 e_3=e_4$  & $e_2 e_1=e_3+e_4$ & $e_2 e_3=e_4$ & $e_3 e_1=e_4$ & $e_3 e_2=e_4$\\
        ${\mathcal J}_{08}$&$:$& $e_1 e_2 = e_3$ & $e_1 e_3= e_4$ & $e_2 e_1=e_3+e_4$ & $e_2 e_2=e_4$ & $e_3 e_1=e_4$\\
        ${\mathcal J}_{09}$&$:$& $e_1 e_2 = e_3$ & $e_1 e_3= e_4$ & $e_2 e_1=e_3+e_4$ & $e_3 e_1=e_4$\\
        ${\mathcal J}_{10}$&$:$& $e_1 e_2 = e_3$ & $e_1 e_3= e_4$ & $e_2 e_1=e_3$ & $e_2 e_3=e_4$ & $e_3 e_1=e_4$  & $e_3 e_2=e_4$\\
        ${\mathcal J}_{11}$&$:$& $e_1 e_2 = e_3$ & $e_1 e_3= e_4$ & $e_2 e_1=e_3$ & $e_3 e_1=e_4$\\
        ${\mathcal J}_{12}$&$:$& $e_1 e_2 = e_3$ & $e_1 e_3= e_4$ & $e_2 e_1=e_3$ & $e_2 e_2=e_4$ & $e_3 e_1=e_4$\\
        ${\mathcal J}_{13}$&$:$& $e_1 e_2 = e_3$ & $e_1 e_3=e_4$  & $e_2 e_1=-e_3$  &$e_3 e_1=-e_4$ \\
        ${\mathcal J}_{14}$&$:$& $e_1 e_1 = e_4$ & $e_1 e_2 = e_3$ & $e_1 e_3=e_4$ & $e_2 e_1=-e_3$ & $e_3 e_1=-e_4$ \\
        ${\mathcal J}_{15}$&$:$& $e_1 e_2 = e_3$ &  $e_1 e_3=e_4$ & $e_2 e_1=-e_3+e_4$ & $e_3 e_1=-e_4$ \\
        ${\mathcal J}_{16}$&$:$& $e_1 e_2 = e_3$ &  $e_1 e_3=e_4$ & $e_2 e_1=-e_3$ & $e_2 e_2=e_4$ & $e_3 e_1=-e_4$ \\
        ${\mathcal J}_{17}$&$:$& $e_1 e_1 = e_4$ & $e_1 e_2 = e_3$ & $e_1 e_3=e_4$ & $e_2 e_1=-e_3$ & $e_2 e_2=e_4$ & $e_3 e_1=-e_4$\\
        ${\mathcal J}_{18}$&$:$& $e_1 e_1 = e_2$ &$e_1 e_2 = e_3$ & $e_1 e_3=e_4$ & $e_2 e_1=e_3$ & $e_2 e_2=e_4$ & $e_3 e_1=e_4$
\end{longtable}

\subsection{The algebraic classification of 
  Kokoris  algebras}\label{Kokoris}
A flexible algebra $\bf A$ is called   {a Kokoris algebra}
if $\bf A^+$ is associative.
Kokoris  algebras were introduced in  \cite{k60} 
and they also appear in \cite{sk91,D24}.
  $({\bf A},\cdot) $ is a Kokoris algebra if
and only if $({\bf A},\circ ,[\cdot,\cdot])$ is a generic Poisson
algebra. Also, observe that Poisson algebras are related to  Lie admissible Kokoris algebras.

\subsubsection{The algebraic classification of 
  complex $3$-dimensional Kokoris  algebras}\label{Kokoris}

\begin{theorem}
\label{3-dim Kokoris algebra}Let $\bf{K}$ be a complex $3$-dimensional Kokoris
algebra. Then $\bf{K}$ is an associative commutative algebra listed in Theorem \ref{asscom}, 
an anticommutative algebra listed in Theorem \ref{3ant} or 
it is isomorphic to one
of the following algebras:

\begin{longtable}{lllllllllllll} 
${\bf A}_{02}^{\alpha\neq0}$ & $:$ & $e_1\cdot e_2=\left( 1+\alpha \right) e_3$ & $e_2\cdot e_1=\left(1-\alpha \right) e_3$ \\
 
${\bf A}_{27}$ & $:$ & $e_1\cdot e_1=e_2$ & $e_1\cdot e_3=e_3$ & $e_3\cdot e_1=-e_3$\\

${\bf A}_{28}$ & $:$ & $e_1\cdot e_1=e_2$ & $e_1\cdot e_3=e_2$ & $e_3\cdot e_1=-e_2$\\

${\bf A}_{29}$ & $:$ & $e_{1}\cdot e_{1}=e_{1}$ & $e_{1}\cdot e_{2}=e_{2}$ & $e_{2}\cdot e_{1}=e_{2}$ & $e_{1}\cdot e_{3}=e_{3}$\\ 
& & $e_{3}\cdot e_{1}=e_{3}$ & $e_{2}\cdot e_{3}=e_{3}$ & $e_{3}\cdot e_{2}=-e_{3}$\\

${\bf A}_{30}$ & $:$ & $e_{1}\cdot e_{1}=e_{1}$ & $e_{2}\cdot e_{3}=e_{3}$ & $e_{3}\cdot e_{2}=-e_{3}$\\

\end{longtable}

All listed algebras are non-isomorphic except:
${\bf A}_{02}^{\alpha }\cong {\bf A}_{02}^{-\alpha}.$ 
\end{theorem}

 \subsubsection{The algebraic classification of 
  complex $3$-dimensional   flexible  power-associative  algebras}\label{Kokoris}

It is a well known fact that if an algebra $\textbf{A}$ is power-associative then $%
\textbf{A}^{+}$ is power-associative. For flexible algebras the converse is also
true \cite{Kleinfeld}. In \cite{Rodrigues}, it is proved that any
commutative power-associative algebra with dimension at most $3$ is a Jordan
algebra. Since a flexible algebra $\textbf{A}$ is a noncommutative Jordan algebra if
and only if the corresponding algebra $\textbf{A}^{+}$ is a Jordan algebra, all
flexible power-associative algebras of dimension at most $3$ are
noncommutative Jordan algebras.

\medskip

\begin{corollary}
Let ${\bf A}$ be a complex $3$-dimensional    flexible  power-associative algebra, 
then ${\bf A}$ is isomorphic to one algebra listed in Theorems \ref{asscom}, \ref{3-dim Jordan}, \ref{3ant} and \ref{3-dim noncommutative Jordan}.
\end{corollary}

\subsubsection{The algebraic classification of 
  complex  $4$-dimensional nilpotent   Kokoris  algebras}\label{Kokorisnilp}

\begin{theorem}\label{koknilp}
Let $\mathcal J$ be a complex $4$-dimensional nilpotent Kokoris   
algebra. Then $\mathcal J$ is a $2$-step nilpotent  algebra listed in  \cite{kppv}  or  isomorphic to one
of the following algebras:
\end{theorem}

\begin{longtable}{llllllll}

        ${\mathcal J}_{01}$&$:$& $e_1 e_1 = e_2$ & $e_1 e_2=e_3$  & $e_2 e_1=e_3$ \\

        ${\mathcal J}_{03}$&$:$& $e_1 e_1 = e_2$ & $e_1 e_2=e_4$  & $e_2 e_1=e_4$ & $e_3 e_1=e_4$ & $e_3 e_3=e_4$\\
        ${\mathcal J}_{04}$&$:$& $e_1 e_1 = e_2$ & $e_1 e_2=e_4$  & $e_2 e_1=e_4$ & $e_3 e_1=e_4$\\

        ${\mathcal J}_{06}$&$:$& $e_1 e_1 = e_2$ & $e_1 e_2=e_4$  & $e_2 e_1=e_4$ & $e_3e_3 =e_4$\\

        ${\mathcal J}_{13}$&$:$& $e_1 e_2 = e_3$ & $e_1 e_3=e_4$  & $e_2 e_1=-e_3$  &$e_3 e_1=-e_4$ \\
        ${\mathcal J}_{14}$&$:$& $e_1 e_1 = e_4$ & $e_1 e_2 = e_3$ & $e_1 e_3=e_4$ & $e_2 e_1=-e_3$ & $e_3 e_1=-e_4$ \\
        ${\mathcal J}_{15}$&$:$& $e_1 e_2 = e_3$ &  $e_1 e_3=e_4$ & $e_2 e_1=-e_3+e_4$ & $e_3 e_1=-e_4$ \\
        ${\mathcal J}_{16}$&$:$& $e_1 e_2 = e_3$ &  $e_1 e_3=e_4$ & $e_2 e_1=-e_3$ & $e_2 e_2=e_4$ & $e_3 e_1=-e_4$ \\
        ${\mathcal J}_{17}$&$:$& $e_1 e_1 = e_4$ & $e_1 e_2 = e_3$ & $e_1 e_3=e_4$ & $e_2 e_1=-e_3$ & $e_2 e_2=e_4$ & $e_3 e_1=-e_4$\\
        ${\mathcal J}_{18}$&$:$& $e_1 e_1 = e_2$ &$e_1 e_2 = e_3$ & $e_1 e_3=e_4$ & $e_2 e_1=e_3$ & $e_2 e_2=e_4$ & $e_3 e_1=e_4$
\end{longtable}

\subsection{The algebraic classification of 
  standard    algebras}
The notions of standard   algebras were introduced in \cite{Alb}. 
An algebra is defined to be  {standard} in case the following two
identities hold: 
\begin{eqnarray*}
(x,y,z)+(z,x,y)-(x,z,y) &=&0, \\
(x,y,wz)+(w,y,xz)+(z,y,wx) &=&0.
\end{eqnarray*}%
Standard algebras include all associative algebras and   Jordan
algebras.
It is proved that the variety of standard algebras is just the minimal variety containing the variety of associative algebras and the variety of Jordan algebras   \cite{hac18}.
Moreover, every standard algebra is a noncommutative Jordan
algebra and is therefore power-associative.
 By some  direct verification of standard    identities in noncommutative Jordan algebras, 
we have the following statements.

\subsubsection{The algebraic classification of 
complex $3$-dimensional standard      algebras}

\begin{theorem}\label{teostan}
Let $\bf{S}$ be a complex $3$-dimensional standard
algebra. Then $\bf{S}$ is a Jordan algebra listed in Theorems \ref{asscom} and \ref{3-dim Jordan} or  isomorphic to one
of the following algebras:
 
\begin{longtable}{lllllll}

${\bf A}_{02}^{\alpha\neq0 }$ & $:$ & $e_1\cdot e_2=\left(1+\alpha \right)e_3$ & $e_2\cdot e_1=\left(1-\alpha \right) e_3$\\

${\bf A}_{13}^{\frac{1}{2}}$ & $:$ & $e_{1}\cdot e_{1}=e_{1}$ & $e_{1}\cdot e_{2}=e_{2}$  
  & $e_{1}\cdot e_{3}=e_{3}$  & $e_{3}\cdot e_{1}=e_{3}$ \\

${\bf A}_{13}^{-\frac{1}{2} }$ & $:$ & $e_{1}\cdot e_{1}=e_{1}$ &  $e_{2}\cdot e_{1}= e_{2}$  
  & $e_{1}\cdot e_{3}=e_{3}$  & $e_{3}\cdot e_{1}=e_{3}$ \\

${\bf A}_{14}^{0,\frac{1}{2}}$ & $:$ & $e_{1}\cdot e_{1}=e_{1}$ & $e_{1}\cdot e_{2}=\frac{1}{2}e_{2}$ & $e_{2}\cdot e_{1}=
\frac{1}{2}e_{2}$ & $e_{1}\cdot e_{3}=e_{3}$\\

${\bf A}_{14}^{0,-\frac{1}{2}}$ & $:$ & $e_{1}\cdot e_{1}=e_{1}$ & $e_{1}\cdot e_{2}=\frac{1}{2}e_{2}$ & $e_{2}\cdot e_{1}=
\frac{1}{2}e_{2}$ & $e_{3}\cdot e_{1}=e_{3}$\\

${\bf A}_{14}^{\frac{1}{2},\frac{1}{2}}$ & $:$ & $e_{1}\cdot e_{1}=e_{1}$ & $e_{1}\cdot e_{2}=e_{2}$ & $e_{1}\cdot e_{3}=e_{3}$\\

${\bf A}_{14}^{-\frac{1}{2},-\frac{1}{2}}$ & $:$ & $e_{1}\cdot e_{1}=e_{1}$ & $e_{2}\cdot e_{1}=e_{2}$ & $e_{3}\cdot e_{1}=e_{3}$\\

${\bf A}_{14}^{\frac{1}{2},-\frac{1}{2}}$ & $:$ & $e_{1}\cdot e_{1}=e_{1}$ & $e_{1}\cdot e_{2}=e_{2}$ & $e_{3}\cdot e_{1}=e_{3}$\\

${\bf A}_{17}^{\frac{1}{2}}$ & $:$ & $e_{1}\cdot e_{1}=e_{1}$ & $e_{2}\cdot e_{2}=e_{2}$ & $e_{1}\cdot e_{3}=e_{3}$ & $e_{3}\cdot e_{2}=e_{3}$\\

${\bf A}_{18}^{\frac{1}{2}}$ & $:$ & $e_{1}\cdot e_{1}=e_{1}$ & $e_{1}\cdot e_{2}=e_{2}$\\

${\bf A}_{18}^{-\frac{1}{2}}$ & $:$ & $e_{1}\cdot e_{1}=e_{1}$ & $e_{2}\cdot e_{1}=e_{2}$\\

${\bf A}_{19}^{\frac{1}{2}}$ & $:$ & $e_{1}\cdot e_{1}=e_{1}$ & $e_{1}\cdot e_{2}=e_{2}$ & $e_{3}\cdot e_{3}=e_{3}$\\

${\bf A}_{19}^{-\frac{1}{2}}$ & $:$ & $e_{1}\cdot e_{1}=e_{1}$ & $e_{2}\cdot e_{1}=e_{2}$ & $e_{3}\cdot e_{3}=e_{3}$\\

${\bf A}_{20}$ & $:$ & $e_{2}\cdot e_{3}=e_{1}$ & $e_{3}\cdot e_{2}=-e_{1}$\\

${\bf A}_{28}$ & $:$ & $e_{1}\cdot
e_{1}=e_{2}$ & $e_{1}\cdot e_{3}=e_{2}$ & $e_{3}\cdot e_{1}=-e_{2}$\\
\end{longtable}

All listed algebras are non-isomorphic except:
${\bf A}_{02}^{\alpha }\cong {\bf A}_{02}^{-\alpha }. $
\end{theorem}

\subsubsection{The algebraic classification of 
  complex  $4$-dimensional nilpotent   standard  algebras}\label{standnilp}

\begin{theorem}
Any complex $4$-dimensional nilpotent noncommutative Jordan algebra, listed in Theorem \ref{nilp4nj},  is a standard algebra. 
\end{theorem}

\section{The geometric classification of 
non-associative   algebras}

\subsection{Definitions and notation}
Given an $n$-dimensional vector space $\mathbb V$, the set ${\rm Hom}(\mathbb V \otimes \mathbb V,\mathbb V) \cong \mathbb V^* \otimes \mathbb V^* \otimes \mathbb V$ is a vector space of dimension $n^3$. This space has the structure of the affine variety $\mathbb{C}^{n^3}.$ Indeed, let us fix a basis $e_1,\dots,e_n$ of $\mathbb V$. Then any $\mu\in {\rm Hom}(\mathbb V \otimes \mathbb V,\mathbb V)$ is determined by $n^3$ structure constants $c_{ij}^k\in\mathbb{C}$ such that
$\mu(e_i\otimes e_j)=\sum\limits_{k=1}^nc_{ij}^ke_k$. A subset of ${\rm Hom}(\mathbb V \otimes \mathbb V,\mathbb V)$ is {\it Zariski-closed} if it can be defined by a set of polynomial equations in the variables $c_{ij}^k$ ($1\le i,j,k\le n$).

Let $T$ be a set of polynomial identities.
The set of algebra structures on $\mathbb V$ satisfying polynomial identities from $T$ forms a Zariski-closed subset of the variety ${\rm Hom}(\mathbb V \otimes \mathbb V,\mathbb V)$. We denote this subset by $\mathbb{L}(T)$.
The general linear group ${\rm GL}(\mathbb V)$ acts on $\mathbb{L}(T)$ by conjugations:
$$ (g * \mu )(x\otimes y) = g\mu(g^{-1}x\otimes g^{-1}y)$$
for $x,y\in \mathbb V$, $\mu\in \mathbb{L}(T)\subset {\rm Hom}(\mathbb V \otimes\mathbb V, \mathbb V)$ and $g\in {\rm GL}(\mathbb V)$.
Thus, $\mathbb{L}(T)$ is decomposed into ${\rm GL}(\mathbb V)$-orbits that correspond to the isomorphism classes of algebras.
Let ${\mathcal O}(\mu)$ denote the orbit of $\mu\in\mathbb{L}(T)$ under the action of ${\rm GL}(\mathbb V)$ and $\overline{{\mathcal O}(\mu)}$ denote the Zariski closure of ${\mathcal O}(\mu)$.

Let $\bf A$ and $\bf B$ be two $n$-dimensional algebras satisfying the identities from $T$, and let $\mu,\lambda \in \mathbb{L}(T)$ represent $\bf A$ and $\bf B$, respectively.
We say that $\bf A$ degenerates to $\bf B$ and write $\bf A\to \bf B$ if $\lambda\in\overline{{\mathcal O}(\mu)}$.
Note that in this case we have $\overline{{\mathcal O}(\lambda)}\subset\overline{{\mathcal O}(\mu)}$. Hence, the definition of degeneration does not depend on the choice of $\mu$ and $\lambda$. If $\bf A\not\cong \bf B$, then the assertion $\bf A\to \bf B$ is called a {\it proper degeneration}. We write $\bf A\not\to \bf B$ if $\lambda\not\in\overline{{\mathcal O}(\mu)}$.

Let $\bf A$ be represented by $\mu\in\mathbb{L}(T)$. Then  $\bf A$ is  {\it rigid} in $\mathbb{L}(T)$ if ${\mathcal O}(\mu)$ is an open subset of $\mathbb{L}(T)$.
 Recall that a subset of a variety is called irreducible if it cannot be represented as a union of two non-trivial closed subsets.
 A maximal irreducible closed subset of a variety is called an {\it irreducible component}.
It is well known that any affine variety can be represented as a finite union of its irreducible components in a unique way.
The algebra $\bf A$ is rigid in $\mathbb{L}(T)$ if and only if $\overline{{\mathcal O}(\mu)}$ is an irreducible component of $\mathbb{L}(T)$.

\medskip

\noindent {\bf Method of the description of degenerations of algebras.} In the present work we use the methods applied to Lie algebras in \cite{GRH}.
First of all, if $\bf A\to \bf B$ and $\bf A\not\cong \bf B$, then $\mathfrak{Der}(\bf A)<\mathfrak{Der}(\bf B)$, where $\mathfrak{Der}(\bf A)$ is the   algebra of derivations of $\bf A$. We compute the dimensions of algebras of derivations and check the assertion $\bf A\to \bf B$ only for such $\bf A$ and $\bf B$ that $\mathfrak{Der}(\bf A)<\mathfrak{Der}(\bf B)$.

To prove degenerations, we construct families of matrices parametrized by $t$. Namely, let $\bf A$ and $\bf B$ be two algebras represented by the structures $\mu$ and $\lambda$ from $\mathbb{L}(T)$ respectively. Let $e_1,\dots, e_n$ be a basis of $\mathbb  V$ and $c_{ij}^k$ ($1\le i,j,k\le n$) be the structure constants of $\lambda$ in this basis. If there exist $a_i^j(t)\in\mathbb{C}$ ($1\le i,j\le n$, $t\in\mathbb{C}^*$) such that $E_i^t=\sum\limits_{j=1}^na_i^j(t)e_j$ ($1\le i\le n$) form a basis of $\mathbb V$ for any $t\in\mathbb{C}^*$, and the structure constants of $\mu$ in the basis $E_1^t,\dots, E_n^t$ are such rational functions $c_{ij}^k(t)\in\mathbb{C}[t]$ that $c_{ij}^k(0)=c_{ij}^k$, then $\bf A\to \bf B$.
In this case  $E_1^t,\dots, E_n^t$ is called a {\it parametrized basis} for $\bf A\to \bf B$.
In  case of  $E_1^t, E_2^t, \ldots, E_n^t$ is a {\it parametric basis} for ${\bf A}\to {\bf B},$ it will be denoted by
${\bf A}\xrightarrow{(E_1^t, E_2^t, \ldots, E_n^t)} {\bf B}$. 
To simplify our equations, we will use the notation $A_i=\langle e_i,\dots,e_n\rangle,\ i=1,\ldots,n$ and write simply $A_pA_q\subset A_r$ instead of $c_{ij}^k=0$ ($i\geq p$, $j\geq q$, $k< r$).

%If the number of orbits under the action of $GL(\mathbb V)$ on  $\mathbb{L}(T)$ is finite, then the constructions of some %degenerations and some non-degenerations give the description of all rigid algebras and irreducible components.

Let ${\bf A}(*):=\{ {\bf A}(\alpha)\}_{\alpha\in I}$ be a series of algebras, and let $\bf B$ be another algebra. Suppose that for $\alpha\in I$, $\bf A(\alpha)$ is represented by the structure $\mu(\alpha)\in\mathbb{L}(T)$ and $\bf B$ is represented by the structure $\lambda\in\mathbb{L}(T)$. Then we say that $\bf A(*)\to \bf B$ if $\lambda\in\overline{\{{\mathcal O}(\mu(\alpha))\}_{\alpha\in I}}$, and $\bf A(*)\not\to \bf B$ if $\lambda\not\in\overline{\{{\mathcal O}(\mu(\alpha))\}_{\alpha\in I}}$.

Let $\bf A(*)$, $\bf B$, $\mu(\alpha)$ ($\alpha\in I$) and $\lambda$ be as above. To prove $\bf A(*)\to \bf B$ it is enough to construct a family of pairs $(f(t), g(t))$ parametrized by $t\in\mathbb{C}^*$, where $f(t)\in I$ and $g(t)\in {\rm GL}(\mathbb V)$. Namely, let $e_1,\dots, e_n$ be a basis of $\mathbb V$ and $c_{ij}^k$ ($1\le i,j,k\le n$) be the structure constants of $\lambda$ in this basis. If we construct $a_i^j:\mathbb{C}^*\to \mathbb{C}$ ($1\le i,j\le n$) and $f: \mathbb{C}^* \to I$ such that $E_i^t=\sum\limits_{j=1}^na_i^j(t)e_j$ ($1\le i\le n$) form a basis of $\mathbb V$ for any  $t\in\mathbb{C}^*$, and the structure constants of $\mu({f(t)})$ in the basis $E_1^t,\dots, E_n^t$ are such rational functions $c_{ij}^k(t)\in\mathbb{C}[t]$ that $c_{ij}^k(0)=c_{ij}^k$, then $\bf A(*)\to \bf B$. In this case  $E_1^t,\dots, E_n^t$ and $f(t)$ are called a parametrized basis and a {\it parametrized index} for $\bf A(*)\to \bf B$, respectively.

We now explain how to prove $\bf A(*)\not\to\mathcal  \bf B$.
Note that if $\mathfrak{Der} \ \bf A(\alpha)  > \mathfrak{Der} \  \bf B$ for all $\alpha\in I$ then $\bf A(*)\not\to\bf B$.
One can also use the following  Lemma, whose proof is the same as the proof of Lemma 1.5 from \cite{GRH}.

\begin{lemma}\label{gmain}
Let $\mathfrak{B}$ be a Borel subgroup of ${\rm GL}(\mathbb V)$ and $\mathcal{R}\subset \mathbb{L}(T)$ be a $\mathfrak{B}$-stable closed subset.
If $\bf A(*) \to \bf B$ and for any $\alpha\in I$ the algebra $\bf A(\alpha)$ can be represented by a structure $\mu(\alpha)\in\mathcal{R}$, then there is $\lambda\in \mathcal{R}$ representing $\bf B$.
\end{lemma}

\subsection{The geometric classification of   Kokoris algebras}
The main result of the present section is the following theorems.

\begin{theorem}\label{geo1}
The variety of complex $3$-dimensional Kokoris    algebras  has 
dimension  $9$   and it has  $5$  irreducible components defined by  
\begin{center}
$\mathcal{C}_1=\overline{\mathcal{O}( {\bf A}_{02}^{\alpha})},$ \
$\mathcal{C}_2=\overline{\mathcal{O}( {\bf A}_{04})},$ \ 
$\mathcal{C}_3=\overline{\mathcal{O}( {\bf A}_{24}^{\alpha})},$ \  
$\mathcal{C}_4=\overline{\mathcal{O}( {\bf A}_{29})},$ \
$\mathcal{C}_5=\overline{\mathcal{O}( {\bf A}_{30})}.$ 
\end{center}
In particular, there are only three rigid algebras in this variety.
 
\end{theorem}

\begin{proof}
After carefully  checking  the dimensions of orbit closures of the more important for us algebras, we have 

\begin{longtable}{rcl}
      
$\dim  \mathcal{O}({\bf A}_{04})=\dim  \mathcal{O}({\bf A}_{24}^{\alpha})$&$=$&$9,$ \\ 
$\dim \mathcal{O}({\bf A}_{29})=
\dim \mathcal{O}({\bf A}_{30})$&$=$&$7,$ \\ 

$\dim \mathcal{O}({\bf A}_{02}^{\alpha})$&$=$&$ 6.$\\

\end{longtable}

${\bf A}_{04}$ is  commutative and ${\bf A}_{24}^\alpha$ is anticommutative, 
hence $\{ {\bf A}_{04}, \  {\bf A}_{24}^\alpha \} \not\to \{ {\bf A}_{02}^\alpha,\  {\bf A}_{29},\  {\bf A}_{30} \}.$
The principal non-degenerations 
$\{ {\bf A}_{29},\  \ {\bf A}_{30} \}  \not\to   {\bf A}_{02}^{\alpha}$  are justified by the following condition
\begin{center} 
$\mathcal R=\left\{ 
c_{12}^3=-c_{21}^3, \ 
c_{13}^2=c_{31}^2=0, \ c_{23}^1=c_{32}^1=0 
\right\}.
$\end{center}

All necessary degenerations are given below 
\begin{enumerate}
    \item  ${\bf A}_{30} \xrightarrow{ (te_1+e_2-e_3, -te_2+te_3, t^{2}e_3)} {\bf A}_{27},$ \  
${\bf A}_{30} \xrightarrow{ (te_1+e_2, -te_2, -te_3)} {\bf A}_{28}.$
\item Thanks to \cite{ikv20},  
${\bf A}_{24}^{\alpha} \to  {\bf A}_{20},\  {\bf A}_{21}, \ {\bf A}_{22}^\alpha,\  {\bf A}_{23},\  {\bf A}_{25}, \  {\bf A}_{26}.$ 
\item Thanks to \cite{gkk}, 
${\bf A}_{04} \to {\bf A}_{01}, \  {\bf A}_{03}, \  {\bf A}_{05},\  {\bf A}_{06},  \ {\bf A}_{07}, \ {\bf A}_{08}, \  {\bf A}_{09}, \  {\bf A}_{10}, \  {\bf A}_{11}.$

\end{enumerate}

\end{proof}

\begin{theorem}\label{geo1nilp}
The variety of complex $4$-dimensional nilpotent Kokoris    algebras  has 
dimension  $13$   and it has  $5$  irreducible components defined by  
\begin{center}
$\mathcal{C}_1=\overline{\mathcal{O}( {\mathcal J}_{03})},$ \  
$\mathcal{C}_2=\overline{\mathcal{O}( {\mathcal J}_{17})},$ \
$\mathcal{C}_3=\overline{\mathcal{O}( {\mathcal J}_{18})},$ 
$\mathcal{C}_3=\overline{\mathcal{O}(\mathfrak{N}_2(\alpha))},$ \
$\mathcal{C}_5=\overline{\mathcal{O}( \mathfrak{N}_3(\alpha))}.$ \ 
\end{center}
In particular, there are only three rigid algebras in this variety.
 
\end{theorem}

\begin{proof}
All necessary degenerations are given below 
\begin{enumerate}
    \item  
    ${\mathcal J}_{03} \xrightarrow{ (t^{-2}e_1, t^{-4}e_2, t^{-3}e_3, t^{-6}e_4)} {\mathcal J}_{06},$   
${\mathcal J}_{18} \xrightarrow{ (e_1, e_2, e_3, t^{-1}e_4)} {\mathcal J}_{01}.$   

\item Thanks to \cite{ikv18}, 
${\mathcal J}_{17} \to {\mathcal J}_{14}, \  {\mathcal J}_{15}, \  {\mathcal J}_{16};$
${\mathcal J}_{03} \to {\mathcal J}_{04};$ ${\mathcal J}_{14} \to {\mathcal J}_{13}.$

\end{enumerate}
Recall that the geometric classification of $4$-dimensional noncommutative Jordan algebras was given in \cite{ikv18}. Hence, ${\mathcal J}_{17},$ ${\mathcal J}_{18}$ and 
$$\begin{array}{lllllll}
\mathfrak{N}_2(\alpha)  & e_1e_1 = e_3, &e_1e_2 = e_4,  &e_2e_1 = -\alpha e_3, &e_2e_2 = -e_4 \\
\mathfrak{N}_3(\alpha)  & e_1e_1 = e_4, &e_1e_2 = \alpha e_4,  &e_2e_1 = -\alpha e_4, &e_2e_2 = e_4,  &e_3e_3 = e_4,
\end{array}$$
also, give irreducible components in the variety of  $4$-dimensional nilpotent Kokoris algebras.
Since, $\mathfrak{Der}({\mathcal J}_{03})=3$ and  $\mathfrak{Der}({\mathcal J}_{17})=4$,  we have ${ \mathcal J}_{17} \not\to {\mathcal J}_{03}.$
Since, ${\mathcal J}_{18}$ is commutative and ${\mathcal J}_{03}$ is noncommutative,  we have ${ \mathcal J}_{18} \not\to {\mathcal J}_{03}.$

\end{proof}

\subsection{The geometric classification of  standard algebras}
The main result of the present section is the following theorem.

\begin{theorem}\label{geo2}
The variety of complex  $3$-dimensional standard  algebras  has 
dimension  $9$   and it has  $14$  irreducible components defined by  
\begin{center}
$\mathcal{C}_1=\overline{ \mathcal{O}({\bf A}_{02}^{\alpha})},$ \
$\mathcal{C}_2=\overline{ \mathcal{O}({\bf A}_{04})},$ \
$\mathcal{C}_3=\overline{ \mathcal{O}({\bf A}_{12})},$ \ 
$\mathcal{C}_4=\overline{ \mathcal{O}({\bf A}_{14}^{0,0})},$ \
$\mathcal{C}_5=\overline{ \mathcal{O}({\bf A}_{14}^{\frac12,\frac12})},$ \\
$\mathcal{C}_6=\overline{ \mathcal{O}({\bf A}_{14}^{-\frac12, -\frac12})},$ \ 
$\mathcal{C}_7=\overline{ \mathcal{O}({\bf A}_{14}^{0,\frac12})},$ \  
$\mathcal{C}_8=\overline{ \mathcal{O}({\bf A}_{14}^{0,-\frac12})},$ \
$\mathcal{C}_9=\overline{ \mathcal{O}({\bf A}_{14}^{\frac12,-\frac12})},$\
$\mathcal{C}_{10}=\overline{\mathcal{O}( {\bf A}_{16})},$ \\  
$\mathcal{C}_{11}=\overline{\mathcal{O}( {\bf A}_{17}^{\frac12})},$ \
$\mathcal{C}_{12}=\overline{\mathcal{O}( {\bf A}_{19}^{0})},$ \
$\mathcal{C}_{13}=\overline{\mathcal{O}( {\bf A}_{19}^{\frac12})},$ \ 
$\mathcal{C}_{14}=\overline{\mathcal{O}( {\bf A}_{19}^{-\frac12})}.$ \

\end{center}
In particular, there are only thirteen rigid algebras in this variety.
 
\end{theorem}

\begin{proof}

Thanks to \cite{gkk}, we have a description of all degenerations between $3$-dimensional Jordan algebras.
Hence,  algebras 
${\bf A}_{04},$ ${\bf A}_{12},$ ${\bf A}_{14}^{0,0},$ ${\bf A}_{16}$ and ${\bf A}_{19}^{0}$ are rigid in the variety of $3$-dimensional Jordan algebras, 
but algebras ${\bf A}_{01},$ ${\bf A}_{03},$ ${\bf A}_{05},$ ${\bf A}_{06},$ ${\bf A}_{07},$ ${\bf A}_{08},$ ${\bf A}_{09},$ ${\bf A}_{10},$ ${\bf A}_{11},$ ${\bf A}_{13}^0, {\bf A}_{15},$ ${\bf A}_{17}^0$ and ${\bf A}_{18}^0$ are not rigid in the variety of $3$-dimensional standard algebras.
Let us give some trivial degenerations:
\begin{center}
${\bf A}_{02}^{\frac 1 t} \xrightarrow{ (e_1+e_2, 2e_3, -t e_1+te_2)} {\bf A}_{28},$ \ 
${\bf A}_{19}^{\alpha} \xrightarrow{ (e_1+e_3, e_3,  t e_2)} {\bf A}_{13}^{\alpha},$ \
${\bf A}_{19}^{\alpha} \xrightarrow{ (e_1, e_3, t e_2)} {\bf A}_{18}^{\alpha},$ \ 
${\bf A}_{28} \xrightarrow{ (te_2, te_1, e_3)} {\bf A}_{20}.$ 

\end{center}

After carefully  checking  the dimensions of orbit closures of the more important for us algebras, we have 

\begin{longtable}{rcl}
      
$\dim  \mathcal{O}({\bf A}_{04})$&$=$&$9,$ \\ 

$\dim  \mathcal{O}({\bf A}_{12})$&$=$&$8,$ \\ 

$\dim \mathcal{O}({\bf A}_{16})=
\dim \mathcal{O}({\bf A}_{17}^{\frac12})=%\dim \mathcal{O}({\bf A}_{17}^{-\frac12})=
\dim \mathcal{O}({\bf A}_{19}^0)=\dim \mathcal{O}({\bf A}_{19}^{\frac12})=\dim \mathcal{O}({\bf A}_{19}^{-\frac12})$&$=$&$7,$ \\ 

$ \dim \mathcal{O}({\bf A}_{02}^{\alpha})$&$=$&$ 6,$\\

$\dim \mathcal{O}({\bf A}_{14}^{0,\frac12})=
\dim \mathcal{O}({\bf A}_{14}^{0,-\frac12})=
\dim \mathcal{O}({\bf A}_{14}^{\frac12,-\frac12})$&$=$&$5,$
%\dim \mathcal{O}({\bf A}_{28})$&$=
 \\ 

$\dim \mathcal{O}({\bf A}_{14}^{0,0})=
\dim \mathcal{O}({\bf A}_{14}^{\frac12,\frac12})=
\dim \mathcal{O}({\bf A}_{14}^{-\frac12,-\frac12})$&$=$&$3.$ \\ 

\end{longtable}

Algebras ${\bf A}_{04},$ ${\bf A}_{12},$  ${\bf A}_{16}$ and ${\bf A}_{19}^{0}$
are commutative.
Hence,
\begin{center}
    
$\{{\bf A}_{04},$ ${\bf A}_{12},$  ${\bf A}_{16},$  ${\bf A}_{19}^{0} \}$
$\not\to$  
$\big\{{\bf A}_{19}^{\frac12},$ 
${\bf A}_{19}^{-\frac12},$ 
${\bf A}_{17}^{\frac12},$
${\bf A}_{14}^{0,\frac12},$
${\bf A}_{14}^{0,-\frac12},$
${\bf A}_{14}^{\frac12,-\frac12},$
${\bf A}_{14}^{\frac12,\frac12},$
${\bf A}_{14}^{-\frac12,-\frac12},$
${\bf A}_{02}^{\alpha}\big\}.$ 
\end{center}

Below we have listed all the important reasons for necessary non-degenerations.

\begin{longtable}{lcl|l}
\hline
    \multicolumn{4}{c}{Non-degenerations reasons} \\
\hline

$\begin{array}{llll}
{\bf A}_{17}^{\frac12} \\ 
\end{array}$ & $\not \rightarrow  $ & 

$\begin{array}{llll}
{\bf A}_{02}^{\alpha},  \
{\bf A}_{14}^{0,-\frac12}, \ {\bf A}_{14}^{0,\frac12}, \\ 
{\bf A}_{14}^{\frac12,-\frac12},\ {\bf A}_{14}^{0,0}, \\ 
{\bf A}_{14}^{\frac12,\frac12},\ {\bf A}_{14}^{-\frac12,-\frac12}
\end{array}$ 
& 
$\mathcal R=\left\{\begin{array}{lllll}
 A_1A_2+A_2A_1\subset A_2, \  
 A_1A_3+A_3A_1 \subset A_3, \\ 
 
c_{22}^1=c_{22}^3=0, \
c_{21}^3=c_{12}^3, \
c_{12}^2 c_{12}^3 = c_{11}^3 c_{22}^2, \\ 
c_{31}^3 = c_{12}^2 =c_{21}^2, \
c_{13}^3 = c_{11}^1

\end{array}\right\}
$\\

\hline

$\begin{array}{llll}
{\bf A}_{19}^{\frac12},  
{\bf A}_{19}^{-\frac12}
\end{array}$ & $\not \rightarrow  $ & 

$\begin{array}{llll}
{\bf A}_{02}^{\alpha},  \
{\bf A}_{14}^{0,-\frac12}, \ {\bf A}_{14}^{0,\frac12},\\
{\bf A}_{14}^{\frac12,-\frac12},\
{\bf A}_{14}^{0,0},\\  {\bf A}_{14}^{\frac12,\frac12}, \ {\bf A}_{14}^{-\frac12,-\frac12}
\end{array}$ 
& 
$\mathcal R=\left\{\begin{array}{lllll}
 A_1A_2+A_2A_1\subset A_2, \  
 A_1A_3+A_3A_1 \subset A_3, \\ 
 
 A_2A_2\subset A_3, \
c_{32}^3 = c_{23}^3, \ 
c_{31}^3 = c_{13}^3, \\

 c_{21}^3 + c_{12}^2=c_{11}^1, \ 
(c_{23}^3)^2 = c_{22}^2 c_{33}^3, \\

 (c_{13}^3)^2 =c_{11}^1  c_{13}^3 + c_{11}^3 c_{33}^3,\
c_{11}^3 c_{22}^3 = c_{21}^3 c_{12}^3 \\

\end{array}\right\}
$\\

\hline

$\begin{array}{llll}
{\bf A}_{02}^{\alpha}\\
\end{array}$ & $\not \rightarrow  $ & 

$\begin{array}{llll}
{\bf A}_{14}^{0,0}, \  {\bf A}_{14}^{\frac12,\frac12}, \\
{\bf A}_{14}^{-\frac12,-\frac12}, \ {\bf A}_{14}^{0,-\frac12},\\ 
{\bf A}_{14}^{0,\frac12}, \ {\bf A}_{14}^{\frac12,-\frac12}
\end{array}$ 

& 
$\mathcal R=\left\{  
A_1^3=0   
 \right\}
$\\

\hline

$\begin{array}{llll}
{\bf A}_{14}^{A,B\neq A}
\end{array}$ & $\not \rightarrow  $ & 

$\begin{array}{llll}
{\bf A}_{14}^{\alpha,\alpha} 
\end{array}$ 
& 
$\mathcal R=\left\{\begin{array}{lllll}
A_1A_2+A_2A_1\subset A_2, \  
A_2^2=0, \
 c_{11}^2=c_{11}^3=0, \\

 (2 c_{12}^2- c_{11}^1( 1+ 2 A)) (2 c_{12}^2-c_{11}^1(1 + 2 B )) = 
 - 4 c_{13}^2 c_{12}^3, \\
c_{13}^3 +c_{12}^2 = c_{11}^1(1 + A  + B  )

\end{array}\right\}
$\\

\hline

\end{longtable}

\end{proof}

\subsection{The geometric classification of   noncommutative Jordan  algebras}
The main result of the present section is the following theorem.

\begin{theorem}\label{geo3}
The variety of complex $3$-dimensional noncommutative Jordan   algebras  has 
dimension  $9$   and it has  $7$  irreducible components defined by  
\begin{center}
$\mathcal{C}_1=\overline{\mathcal{O}( {\bf A}_{04})},$ \
$\mathcal{C}_2=\overline{\mathcal{O}( {\bf A}_{12})},$ \ 
$\mathcal{C}_3=\overline{\mathcal{O}( {\bf A}_{16})},$ \
$\mathcal{C}_4=\overline{\mathcal{O}( {\bf A}_{17}^{\alpha})},$ \\
$\mathcal{C}_5=\overline{\mathcal{O}( {\bf A}_{19}^{\alpha})},$ \
$\mathcal{C}_6=\overline{\mathcal{O}( {\bf A}_{24}^{\alpha})},$ \ 
%$\mathcal{C}_7=\overline{\mathcal{O}( {\bf A}_{30})},$ \
$\mathcal{C}_7=\overline{\mathcal{O}( {\bf A}_{32})}.$ \\

\end{center}
In particular, there are only four rigid algebras in this variety.
 
\end{theorem}

\begin{proof}

After carefully  checking  the dimensions of orbit closures of the more important for us algebras, we have 

\begin{longtable}{rcl}
      
$\dim  \mathcal{O}({\bf A}_{04})=\dim  \mathcal{O}({\bf A}_{24}^{\alpha})=\dim  \mathcal{O}({\bf A}_{32})$&$=$&$9,$ \\ 
$\dim  \mathcal{O}({\bf A}_{12})=\dim  \mathcal{O}({\bf A}_{17}^{\alpha})=\dim  \mathcal{O}({\bf A}_{19}^{\alpha})$&$=$&$8,$ \\ 

$\dim  \mathcal{O}({\bf A}_{16})$&$=$&$7.$ \\ 

\end{longtable}

Below we have listed all the important reasons for necessary non-degenerations.

${\bf A}_{04}$ and ${\bf A}_{12}$ are commutative, 
${\bf A}_{24}^{\alpha}$ are anticommutative, 
${\bf A}_{32},$   ${\bf A}_{19}^{\alpha}$ and  ${\bf A}_{17}^{\alpha}$ 
are noncommutative and non-anticommutative.
Hence, 

\begin{center}
    $\{ {\bf A}_{04},$  ${\bf A}_{12},$ ${\bf A}_{24}^{\alpha}\}$ $\not\to$ 
$ \{{\bf A}_{32},$    ${\bf A}_{19}^{\alpha},$   ${\bf A}_{17}^{\alpha} \}.$

\end{center}

\begin{longtable}{lcl|l}
\hline
    \multicolumn{4}{c}{Non-degenerations reasons} \\
\hline

${\bf A}_{17}^{\alpha}$ & $\not \rightarrow  $ & 

$\begin{array}{llll}
 {\bf A}_{16}  \\
\end{array}$ 
& 
$\mathcal R=\left\{\begin{array}{lllll}
A_2^2 \subset A_2,   \  A_1A_3+A_3A_1\subset A_3, \
c_{22}^3=0, \ c_{23}^3 +c_{32}^3=c_{22}^2

\end{array}\right\}
$\\

\hline

${\bf A}_{19}^{\alpha}$ & $\not \rightarrow  $ & 

$\begin{array}{llll}
{\bf A}_{16} \\
\end{array}$ 
& 
$\mathcal R=\left\{\begin{array}{lllll}
A_1A_2+A_2A_1\subset A_2, \  A_1A_3+A_3A_1\subset A_3, \\ 
\ A_2A_3+A_3A_2=0, \
c_{31}^3+c_{13}^3=c_{11}^1
\end{array}\right\}
$\\

\hline

${\bf A}_{32}$ & $\not \rightarrow  $ & 

$\begin{array}{llll}
{\bf A}_{12}, \ {\bf A}_{16},  \\
{\bf A}_{17}^{\alpha}, \
{\bf A}_{19}^{\alpha}   \\
\end{array}$ 
& 
$\mathcal R=\left\{\begin{array}{lllll}
A_2^2 \subset A_2,   \ 
c_{23}^2=-c_{32}^2, \
c_{23}^3=-c_{32}^3,\
c_{22}^2=0, \
c_{22}^3=0, \
c_{33}^2=0, \
c_{33}^3=0

\end{array}\right\}
$\\

\hline

\end{longtable}

All necessary degenerations are given below: 
\begin{enumerate}

\item 
${\bf A}_{19}^{\frac{\alpha}{2}} \xrightarrow{ (te_1, -2t^2e_2+e_3, t^3e_2)} {\bf A}_{02}^{\alpha},$ \
${\bf A}_{19}^{\alpha} \xrightarrow{ (e_1+e_2,  e_3, te_2)} {\bf A}_{13}^{\alpha},$ \\
{${\bf A}_{33}^{\beta} \xrightarrow{ (e_1+(\beta-\alpha)e_2-\alpha e_3, t^2e_3, te_2+te_3)}  {\bf A}_{14}^{\alpha,\beta},$}  \
${\bf A}_{19}^{\alpha} \xrightarrow{ (e_1, e_3, te_2)} {\bf A}_{18}^{\alpha},$ \
${\bf A}_{17}^{\frac{-2-t}{2t}} \xrightarrow{ (e_1+e_2, te_2, e_3)} {\bf A}_{29},$ \
${\bf A}_{19}^{-\frac 12 +\frac 1t} \xrightarrow{ ( e_2, te_1, e_3)} {\bf A}_{30},$ \
${\bf A}_{33}^{\frac{2\alpha-t}{2+2t}} \xrightarrow{ ((1+t)e_1-(\alpha+\frac{t}{2})e_3, -t^2e_3, e_2+te_3)} {\bf A}_{31}^{\alpha},$ \
${\bf A}_{32} \xrightarrow{ (e_1-\alpha e_3, t e_3, (\alpha+t^2)e_2+e_3)} {\bf A}_{33}^{\alpha}.$

\item  Thanks to Theorem \ref{geo1}, 
all irreducible components of the variety of $3$-dimensional Kokoris algebras are defined by the following
algebras
${\bf A}_{02}^{\alpha},$ \
${\bf A}_{04},$ \
$ {\bf A}_{24}^{\alpha},$ \  
${\bf A}_{29},$  and
$  {\bf A}_{30}.$
Hence, 
${\bf A}_{01},$
${\bf A}_{03},$
${\bf A}_{05},$
${\bf A}_{06},$
${\bf A}_{07},$
${\bf A}_{08},$
${\bf A}_{09},$
${\bf A}_{10},$
${\bf A}_{11},$
${\bf A}_{20},$
${\bf A}_{21},$
${\bf A}_{22}^\alpha,$
${\bf A}_{23},$
${\bf A}_{25},$
${\bf A}_{26},$
${\bf A}_{27}$ and 
${\bf A}_{28}$ do not give irreducible components in the variety of noncommutative Jordan algebras.

\item Thanks to \cite{gkk}, all irreducible components the variety of $3$-dimensional Jordan algebras are defined by   ${\bf A}_{04}, {\bf A}_{12}, {\bf A}_{14}^{0,0}, {\bf A}_{16}, {\bf A}_{19}^0.$
Hence, ${\bf A}_{15}$ does not give an irreducible component in the variety of noncommutative Jordan algebras.

\end{enumerate}

%\begin{longtable}{lcl|lcl} \hline\\  \hline
%${\bf A}_{30}$ & $\xrightarrow{ (te_1-e_2-e_3, te_2+te_3, t^{2}e_3)}$ & ${\bf A}_{27}$ & 
%${\bf A}_{30}$ & $\xrightarrow{ (te_1-e_2, te_2, -te_3)}$ & ${\bf A}_{28}$
%\\  \hline \hline\end{longtable}
 
\end{proof}

%{\bf Conflict of interest:}   There are no competing interests.

%{\bf Data Availibility:} The manuscript has no associated data

%\begin{lemma}
%Let $\mathcal{J}$ be as given in Theorem \ref{3-dim noncommutative Jordan}.
%Then $\dim Der\left( \mathcal{J}\right) \neq 7,8$. Moreover, we have

%\begin{itemize}
%\item $\dim Der\left( \mathcal{J}\right) =0$ if and only if $\mathcal{J\in }%
%\left\{ {\bf A}_{04}, {\bf A}_{32}\right\} $.

%\item $\dim Der\left( \mathcal{J}\right) =1$ if and only if $\mathcal{J\in }%
%\left\{ {\bf A}_{05},{\bf A}_{08}, {\bf A}_{12}, {\bf A}_{24}^{\alpha }\right\}$.

%\item $\dim Der\left( \mathcal{J}\right) =2$ if and only if $\mathcal{J\in }%
%\left\{  {\bf A}_{07}, {\bf A}_{09},  {\bf A}_{11}, {\bf A}
%_{15}, {\bf A}_{16}, {\bf A}_{17}^{\alpha},{\bf A}
%_{19}^{\alpha}, {\bf A}_{25}, {\bf A}_{29}, {\bf A}_{30}, {\bf A}_{33}^{\alpha}\right \}$.

%\item $\dim Der\left( \mathcal{J}\right) =3$ if and only if $\mathcal{J\in }%
%\left\{ {\bf A}_{03}, {\bf A}_{13}^{\alpha }, {\bf A}_{18}^{\alpha}, {\bf A}_{23}, {\bf A}_{26}, {\bf A}_{27}, 
%\right\}$.

%\item $\dim Der\left( \mathcal{J}\right) =4$ if and only if $\mathcal{J\in }%
%\left\{ {\bf A}_{02}^{\alpha }, {\bf A}_{06}, {\bf A}_{10}, {\bf A}
%_{14}^{\alpha ,\beta \neq \alpha }, {\bf A}_{22}^{\alpha }, {\bf A}_{28}, 
 %{\bf A}_{31}^{\alpha}\right\} $.

%\item $\dim Der\left( \mathcal{J}\right) =5$ if and only if $\mathcal{J}={\bf A}_{01}$.

%\item $\dim Der\left( \mathcal{J}\right) =6$ if and only if $\mathcal{J\in }%
%\left\{ {\bf A}_{14}^{\alpha, \alpha}, {\bf A}_{20}, {\bf A}_{21} \right\} $.
%\end{itemize}
%\end{lemma}

\end{document}